\documentclass[a4paper,11 pt,reqno]{amsart} 
\usepackage{a4,amsbsy,amscd,amssymb,amstext,amsmath,mathtools,mathdots,latexsym,oldgerm,enumerate,verbatim,graphicx,ytableau,xy}
\usepackage[hang,flushmargin]{footmisc} 

\makeatletter                                       
\def\l@subsection{\@tocline{2}{0pt}{2.5pc}{2.5pc}{}}
\makeatother                                        

\makeatletter                                                  
\def\chapter{\clearpage\thispagestyle{plain}\global\@topnum\z@ 
\@afterindenttrue \secdef\@chapter\@schapter}
\makeatother

\newtheorem{thm} {Theorem} [section]
\newtheorem{prop}{Proposition} [section]
\newtheorem{lem} {Lemma} [section]

\newtheorem{cornn}{Corollary}

\newtheorem{notnn}{Notation}

\theoremstyle{definition}

\newtheorem{rem} {Remark} [section]
\newtheorem{rems} [rem]{Remarks}
\newtheorem{exa} [rem] {Example}

\newcommand{\mf}{\mathfrak}
\newcommand{\mc}{\mathcal}
\newcommand{\mb}{\mathbb}

\newcommand{\nts}{\negthinspace}     
\newcommand{\Nts}{\nts\nts}

\newcommand{\ov}{\overline}

\newcommand{\ot}{\otimes}           
\newcommand{\la}{\langle}
\newcommand{\ra}{\rangle}

\newcommand{\Hom}{{\rm Hom}}        

\newcommand{\End}{{\rm End}}

\newcommand{\Mat}{{\rm Mat}}

\newcommand{\ind}{{\rm ind}}

\newcommand{\Sym}{{\rm Sym}} 

\newcommand{\id}{{\rm id}}

\let\ttie\t
\newcommand{\tie}[1]{{\let\t\ttie \ttie#1}}
\renewcommand{\t}{\mf{t}}  

\newcommand{\GL}{{\rm GL}}

\newcommand{\ve}{\varepsilon}

\makeatletter
\def\vcdots{\vbox{\baselineskip4\p@ \lineskiplimit\z@
\kern3\p@\hbox{.}\hbox{.}\hbox{.}\Nts\nts\kern3\p@}}
\makeatother

\xyoption{all}

\begin{document}

\title{A combinatorial translation principle and diagram combinatorics for the general linear group}

\begin{abstract}
Let $k$ be an algebraically closed field of characteristic $p>0$. We compute the Weyl filtration multiplicities in indecomposable tilting modules and the decomposition numbers for the general linear group over $k$ in terms of cap diagrams under the assumption that $p$ is bigger than the greatest hook length in the partitions involved.
Then we introduce and study the rational Schur functor from a category of $\GL_n$-modules to the category of modules for the walled Brauer algebra.
As a corollary we obtain the decomposition numbers for the walled Brauer algebra when $p$ is bigger than the greatest hook length in the partitions involved. This is a sequel to an earlier paper on the symplectic group and the Brauer algebra.
\end{abstract}

\author[R.\ Tange]{Rudolf Tange}
\address
{School of Mathematics,
University of Leeds,
LS2 9JT, Leeds, UK}
\email{R.H.Tange@leeds.ac.uk}

\maketitle
\markright{\MakeUppercase{A combinatorial translation principle for} $\GL_n$}

\section{Introduction}
The present paper concerns the general linear group and the walled Brauer algebra, it is a sequel to the paper \cite{LT} where the analogous results for the symplectic group and the Brauer algebra are obtained. For more background we refer to the introduction of \cite{LT}.

The walled Brauer algebra, introduced by Turaev \cite{Tur} and Koike \cite{Koi} and later in \cite{BCHLLS}, is a cellular algebra, see \cite[Thm~2.7]{CdVDM}, and an interesting problem is to determine its decomposition numbers. In characteristic $0$ this was first done in \cite{CdV} in terms of certain cap diagrams.

Let $\GL_n$ be the general linear group over an algebraically closed field $k$ of characteristic $p>0$, and let $V$ be the natural module. In characteristic $0$ there is a well-known relation between certain representations of $\GL_n$ and the representations of the walled Brauer algebra $B_{r,s}(n)$, given by the double centraliser theorem for their actions on $V^{\ot r}\ot(V^*)^{\ot s}$. In characteristic $p$ such a connection doesn't follow from the double centraliser theorem and requires more work. This is done in Section~\ref{s.rat_Schur_functor} of the present paper by means of the rational Schur functor.

We determine the Weyl filtration multiplicities in the indecomposable tilting modules $T(\lambda)$ and the decomposition numbers for the induced modules $\nabla(\lambda)$ of $\GL_n$ when the two partitions that form $\lambda$ have greatest hook length less than $p$. We then introduce the rational Schur functor and use it to obtain from the first multiplicities the decomposition numbers of the walled Brauer algebra under the assumption that $p$ is bigger than the greatest hook length in the partitions involved. Our main tools are the ``reduced" Jantzen Sum Formula, truncation, and translation functors. Our approach is based on the same ideas as \cite{LT}.

The paper is organised as follows. In Section~\ref{s.prelim} we introduce the necessary notation. In Section~\ref{s.JSF} we show that certain terms in the Jantzen Sum Formula may be omitted. This leads to a ``strong linkage principle" in terms of a partial order $\preceq$, and the existence of nonzero homomorphisms between certain pairs of induced modules, see Proposition~\ref{prop.linkage}. Although we do not need our strong linkage principle for the translation functors, we do need it for the truncation. In Section~\ref{s.translation} we prove two basic results about translation that we will use: Propositions~\ref{prop.trans_equivalence} and ~\ref{prop.trans_projection}. They are analogues of the two corresponding results in \cite[Sect~4]{LT} and the proofs are straightforward simplifications of the ones in \cite{LT}.

In Section~\ref{s.arrow_diagrams} we introduce arrow diagrams to represent the weights that satisfy our condition, and we show that the nonzero terms in the reduced Jantzen Sum Formula and the pairs of weights for which we proved the existence of nonzero homomorphisms between the induced modules have a simple description in terms of arrow diagrams, see Lemma~\ref{lem.JSF-arrows}. The arrow diagrams in the present paper are rather different from the ones in \cite{LT}, they should be thought of as circular rather than as a line segment.
As in the case of the symplectic group, the order $\preceq$ and conjugacy under the dot action also have a simple description in terms of the arrow diagram, see Remark~\ref{rems.preceq}.1. In Section~\ref{s.filt_mult} we prove our first main result, Theorem~\ref{thm.filt_mult}, which describes the Weyl filtration multiplicities in certain indecomposable tilting modules in terms of cap diagrams.

In Section~\ref{s.dec_num} we prove our second main result, Theorem~\ref{thm.dec_num}, which describes the decomposition numbers for certain induced modules in terms of cap codiagrams.
In Section~\ref{s.rat_Schur_functor} we introduce the rational Schur functor and determine its basic properties. The main results in this section are Theorem~\ref{thm.rat_schur} and Proposition~\ref{prop.rat_mult}.
As a corollary to Theorem~\ref{thm.filt_mult} and Proposition~\ref{prop.rat_mult} we obtain the decomposition numbers of the walled Brauer algebra under the assumption that $p$ is bigger than the greatest hook length in the partitions involved. 

\section{Preliminaries}\label{s.prelim}
First we recall some general notation from \cite{LT}. Throughout this paper $G$ is a reductive group over an algebraically closed field $k$ of characteristic $p>0$, $T$ is a maximal torus of $G$ and $B^+$ is a Borel subgroup of $G$ containing $T$.
We denote the group of weights relative to $T$, i.e. the group of characters of $T$, by $X$. For $\lambda,\mu\in X$ we write $\mu\le\lambda$ if $\lambda-\mu$ is a sum of positive roots (relative to $B^+$). The Weyl group of $G$ relative to $T$ is denoted by $W$ and the set of dominant weights relative to $B^+$ is denoted by $X^+$. In the category of (rational) $G$-modules, i.e. $k[G]$-comodules, there are several special families of modules. For $\lambda\in X^+$ we have the irreducible $L(\lambda)$ of highest weight $\lambda$, and the induced module $\nabla(\lambda)=\ind_B^Gk_\lambda$, where $B$ is the opposite Borel subgroup to $B^+$ and $k_\lambda$ is the 1-dimensional $B$-module afforded by $\lambda$. The Weyl module and indecomposable tilting module associated to $\lambda$ are denoted by $\Delta(\lambda)$ and $T(\lambda)$. To each $G$-module $M$ we can associate its formal character ${\rm ch}\,M=\sum_{\lambda\in X}\dim M_\lambda e(\lambda)\in(\mb Z X)^W$, where $M_\lambda$ is the weight space associated to $\lambda$ and $e(\lambda)$ is the basis element corresponding to $\lambda$ of the group algebra $\mb Z X$ of $X$ over $\mb Z$. Composition and good or Weyl filtration multiplicities are denoted by $[M:L(\lambda)]$ and $(M:\nabla(\lambda))$ or $(M:\Delta(\lambda))$. For a weight $\lambda$, the character $\chi(\lambda)$ is given by Weyl's character formula \cite[II.5.10]{Jan}. If $\lambda$ is dominant, then ${\rm ch}\,\nabla(\lambda)={\rm ch}\,\Delta(\lambda)=\chi(\lambda)$. The $\chi(\lambda)$, $\lambda\in X^+$, form a $\mb Z$-basis of $(\mb Z X)^W$.
For $\alpha$ a root and $l\in\mb Z$, let $s_{\alpha,l}$ be the affine reflection of $\mb R\ot_{\mb Z}X$ defined by $s_{\alpha,l}(x)=x-a\alpha$, where $a=\la x,\alpha^\vee\ra-lp$. Mostly we replace $\la-,-\ra$ by a $W$-invariant inner product and then the cocharacter group of $T$ is identified with a lattice in $\mb R\ot_{\mb Z}X$ and $\alpha^\vee=\frac{2}{\la\alpha,\alpha\ra}\alpha$.
We have $s_{-\alpha,l}=s_{\alpha,-l}$ and the affine Weyl group $W_p$ is generated by the $s_{\alpha,l}$.
Choose $\rho\in\mathbb Q\ot_{\mb Z}X$ with $\la\rho,\alpha^\vee\ra=1$ for all $\alpha$ simple and define the dot action of $W_p$ on $\mb R\ot_{\mb Z}X$ by $w\cdot x=w(\lambda+\rho)-\rho$. The lattice $X$ is stable under the dot action.
The \emph{linkage principle} \cite[II.6.17,7.2]{Jan} says that if $L(\lambda)$ and $L(\mu)$ belong to the same $G$-block, then $\lambda$ and $\mu$ are $W_p$-conjugate under the dot action. We refer to \cite{Jan} part II for more details.

Unless stated otherwise, $G$ will be the general linear group $\GL_n$. The natural $G$-module $k^n$ is denoted by $V$. We let $T$ be the group of diagonal matrices in $\GL_n$. Then $X$ is naturally identified with $\mb Z^n$ such that the $i$-th diagonal coordinate function corresponds to the $i$-th standard basis element $\ve_i$ of $\mb Z^n$. We let $B^+$ be the Borel subgroup of invertible upper triangular matrices corresponding to the set of positive roots $\ve_i-\ve_j$, $1\le i<j\le n$. Then a weight in $\mb Z^n$ is dominant if and only if it is weakly decreasing.
Such a weight $\lambda$ can uniquely be written as
$$[\lambda^1,\lambda^2]\stackrel{\rm def}{=}(\lambda^1_1,\lambda^1_2,\ldots,0,\ldots,0,\ldots,-\lambda^2_2,-\lambda^2_1)$$
where $\lambda^1=(\lambda^1_1,\lambda^1_2,\ldots)$ and $\lambda^2=(\lambda^2_1,\lambda^2_2,\ldots)$ are partitions with $l(\lambda^1)+l(\lambda^2)\le n$. Here $l(\xi)$ denotes the length of a partition $\xi$. So $X^+$ can be identified with pairs of partitions $(\lambda^1,\lambda^2)$ with $l(\lambda^1)+l(\lambda^2)\le n$.
We will also identify partitions with the corresponding Young diagrams. In explicit examples we will only work with partitions with parts at most $10$ and these may be written in ``exponential form": $(10,7,7,4,2,2,1)$ is denoted by $107^242^21$. For $s_1,s_2\in\{1,\ldots,n\}$ with $s_1+s_2\le n$ we denote the subgroup of $W_p$ generated by the $s_{\alpha,l}$, $\alpha=\ve_i-\ve_j$, $i,j\in\{1,\ldots,s_1,n-s_2+1,\dots,n\}$ by $W_p^{s_1,s_2}$. This is the affine Weyl group of a root system of type $A_{s_1+s_2-1}$. The group $W$ acts on $\mb Z^n$ by permutations, and $W_p\cong W\ltimes pX_0$, where $X_0=\{\lambda\in\mb Z^n\,|\,|\lambda|=0\}$ is the type $A_{n-1}$ root lattice and $|\lambda|=\sum_{i=1}^n\lambda_i$. Note that $W_p^{s_1,s_2}\cong W^{s_1,s_2}\ltimes pX_0^{s_1,s_2}$, where $X_0^{s_1,s_2}$ consists of the vectors in $X_0$ which are $0$ at the positions in $\{s_1+1,\ldots,n-s_2\}$, and $W^{s_1,s_2}=\Sym(\{1,\ldots,s_1,n-s_2+1,\dots,n\})$. We will work with $$\rho=(n,n-1,\ldots,1)\,.$$
It is easy to see that if $\lambda,\mu\in X$ are $W_p$-conjugate and equal at the positions in $\{s_1+1,\ldots,n-s_2\}$, then they are $W_p^{s_1,s_2}$-conjugate. The same applies for the dot action.

To obtain our results we will have to make use of quasihereditary algebras. We refer to \cite[Appendix]{D2} and \cite[Ch~A]{Jan} for the general theory.
For a subset $\Lambda$ of $X^+$ and a $G$-module $M$ we say that $M$ \emph{belongs to $\Lambda$} if all composition factors have highest weight in $\Lambda$ and we denote by $O_\Lambda(M)$ the largest submodule of $M$ which belongs to $\Lambda$. For a quasihereditary algebra one can make completely analogous definitions. We denote the category of $G$-modules which belong to $\Lambda$ by $\mc C_\Lambda$. Any quasihereditary algebra $A$ that we consider will be determined by its labelling set $\Lambda\subseteq X^+$ for the irreducibles, endowed with a suitable partial order. The irreducible, standard/costandard and tilting modules are the irreducible, Weyl/induced and tilting modules for $G$ with the same label: the module category of $A$ is equivalent to $\mc C_\Lambda$.

\section{The reduced Jantzen Sum Formula}\label{s.JSF}
In this section we study the Jantzen Sum Formula for the general linear group $\GL_n$. This is analogous to the results in \cite[Sect~3]{LT} for the symplectic group.
Assume for the moment that $G$ is any reductive group. Jantzen has defined for every Weyl module $\Delta(\lambda)$ of $G$ a descending filtration $\Delta(\lambda)=\Delta(\lambda)^0\supseteq\Delta(\lambda)^1\supseteq\cdots$ such that $\Delta(\lambda)/\Delta(\lambda)^1\cong L(\lambda)$ and $\Delta(\lambda)^i=0$ for $i$ big enough. The Jantzen sum formula \cite[II.8.19]{Jan} relates the formal characters of the $\Delta(\lambda)^i$ with the Weyl characters $\chi(\mu)$, $\mu\in X^+$:
\begin{equation}\label{eq.JSF}
\sum_{i>0}{\rm ch}\,\Delta(\lambda)^i=\sum\nu_p(lp)\chi(s_{\alpha,l}\cdot\lambda)\ ,
\end{equation}
where the sum on the right is over all pairs $(\alpha,l)$, with $l$ an integer $\ge1$ and $\alpha$ a positive root such that $\la\lambda+\rho,\alpha^\vee\ra-lp>0$, and $\nu_p$ is the $p$-adic valuation. Here $\chi(\mu)=0$ if and only if $\la\mu+\rho,\alpha^\vee\ra=0$ for some $\alpha>0$, and if $\chi(\mu)\ne0$, then $\chi(\mu)=\det(w)\chi(w\cdot\mu)$, where $w\cdot\mu$ is dominant for a unique $w\in W$. See \cite[II.5.9(1)]{Jan}. We denote the RHS of \eqref{eq.JSF} by $JSF(\lambda)$.

Now return to our standard assumption $G=\GL_n$. For $\lambda\in X$ we have that $\chi(\lambda)\ne0$ if and only if
\begin{align*}
&(\lambda+\rho)_i\ne(\lambda+\rho)_j\text{\quad for all\ }i,j\in\{1,\ldots,n\}\text{\ with\ }i\ne j.
\end{align*}

We will consider any partition of length $\le n$ as an $n$-tuple, by extending it with zeros and for $\xi\in\mb Z^n$ we denote the reversed tuple by $\breve{\xi}$. So $[\lambda^1,\lambda^2]=\lambda^1-\breve{\lambda^2}$. For $i\in\{1,\ldots,n\}$ we put $i'=n+1-i$. So for $\xi\in\mb Z^n$ we have $\xi_i=\breve{\xi}_{i'}$.

\medskip
\emph{For the remainder of this section $\lambda=[\lambda^1,\lambda^2]\in X^+$ and $\lambda^1$ and $\lambda^2$ are \emph{$p$-cores}, unless stated otherwise.}
\medskip

We will use the following characterisation of $p$-cores. Let $\theta\in\mb Z^m$ with $\theta_{i-1}=\theta_i+1$ for all $i\in\{2,\dots,m\}$. Then a partition $\xi$ with $l(\xi)\le m$ is a $p$-core if and only if for all $i\in\{1,\ldots,m\}$ and all integers $l\ge1$, $(\xi+\theta)_i-lp$ occurs in $\xi+\theta$, provided it is $\ge\theta_m$. This is equivalent to the definition in \cite[Ex~I.1.8]{Mac}.

\begin{lem}\label{lem.zero}
Assume $\alpha=\ve_i-\ve_j$, $1\le i<j\le n$, $\la\lambda+\rho,\alpha^\vee\ra=a+lp$, $a,l>0$, and $\chi(s_{\alpha,l}\cdot\lambda)\ne0$. Then $i\le l(\lambda^1)$ and $j>n-l(\lambda^2)$.
Furthermore, $(\lambda+\rho)_i-a>n-l(\lambda^1)$ and $(\lambda+\rho)_j+a\le l(\lambda^2)$.
\end{lem}
\begin{proof}
We have $(\lambda+\rho)_i-(\lambda+\rho)_j=\la\lambda+\rho,\alpha^\vee\ra=a+lp$. Note that $(\lambda+\rho)_h=(\lambda^1+\rho)_h$ for all $h\le n-l(\lambda^2)$ and $-(\lambda+\rho)\breve{{}_h}=(\lambda^2-\breve{\rho})_h$ for all $h\le n-l(\lambda^1)$. First assume $j\le n-l(\lambda^2)$. Then $(\lambda+\rho)_j=(\lambda^1+\rho)_j$ and $(\lambda+\rho)_i=(\lambda^1+\rho)_i$. Now $(\lambda+\rho)_i-lp=(\lambda+\rho)_j+a=s_{\alpha,l}(\lambda+\rho)_j$ must occur in $\lambda^1+\rho$, clearly strictly between the $i$-th and $j$-th position. So it also occurs in $\lambda+\rho$ strictly between these positions. So $s_{\alpha,l}(\lambda+\rho)$ contains a repeat, contradicting $\chi(s_{\alpha,l}\cdot\lambda)\ne0$.

Next assume $i>l(\lambda^1)$. Then $j'<i'\le n-l(\lambda^1)$. So $-(\lambda+\rho)\breve{\,}_{\nts j'}-lp=-(\lambda+\rho)\breve{\,}_{\nts i'}+a>-\breve{\rho}_{i'}$ must occur in $\lambda^2-\breve{\rho}$ strictly between the $j'$-th and $i'$-th position. So it also must occur in $-(\lambda+\rho)\breve{\,}$ strictly between the $j'$-th and $i'$-th position,
and this means that $s_{\alpha,l}(\lambda+\rho)_i=(\lambda+\rho)_i-a=(\lambda+\rho)_j+lp$ must occur in $\lambda+\rho$ strictly between the $i$-th and $j$-th position.
So $s_{\alpha,l}(\lambda+\rho)$ contains a repeat, contradicting $\chi(s_{\alpha,l}\cdot\lambda)\ne0$.

Now assume $(\lambda+\rho)_i-a\le n-l(\lambda^1)$. Then $(\lambda^2-\breve{\rho})_{j'}-lp=-(\lambda+\rho)_j-lp=-((\lambda+\rho)_i-a)\ge l(\lambda^1)-n=-\breve{\rho}_{n-l(\lambda^1)}$ must occur in $\lambda^2-\breve{\rho}$ strictly between the $j'$-th and $l(\lambda^1)'$-th position. As before this means that $s_{\alpha,l}(\lambda+\rho)_i=(\lambda+\rho)_i-a=(\lambda+\rho)_j+lp$ must occur in $\lambda+\rho$ strictly between the $l(\lambda^1)$-th and $j$-th position. Since $i\le l(\lambda^1)$, $s_{\alpha,l}(\lambda+\rho)$ contains a repeat, contradicting $\chi(s_{\alpha,l}\cdot\lambda)\ne0$.

Finally assume $(\lambda+\rho)_j+a>l(\lambda^2)$. Then $(\lambda^1+\rho)_i-lp=(\lambda+\rho)_i-lp=(\lambda+\rho)_j+a>l(\lambda^2)=\rho_{l(\lambda^2)'}$ must occur in $\lambda^1+\rho$ strictly between the $i$-th and $l(\lambda^2)'$-th position. As before this means that it also occurs in $\lambda+\rho$ strictly between these positions. Since $j\ge l(\lambda^2)'$, $s_{\alpha,l}(\lambda+\rho)$ contains a repeat, contradicting $\chi(s_{\alpha,l}\cdot\lambda)\ne0$.
\end{proof}

By the previous lemma we may, when $\lambda^1$ and $\lambda^2$ are $p$-cores, restrict the sum on the RHS of \eqref{eq.JSF} to the positive roots $\alpha=\ve_i-\ve_j$ with $1\le i\le l(\lambda^1)$ and $n-l(\lambda^2)<j\le n$ (and $\chi(s_{\alpha,l}\cdot\lambda)\ne0$). We will refer to this sum as the \emph{reduced sum} and to the whole equality as the \emph{reduced Jantzen Sum Formula}. For $\mu,\nu\in\mb Z^n$ we write $\mu\subseteq\nu$ when $\mu_i\le\nu_i$ for all $i\in\{1,\ldots,n\}$, and we denote the weakly decreasing permutation of $\mu$ by ${\rm sort}(\mu)$. The next lemma shows that, when working with Weyl characters, the nonzero terms in the reduced sum have distinct Weyl characters.

\begin{lem}\label{lem.inclusion}
Let $\alpha=\ve_i-\ve_j$, $1\le i\le l(\lambda^1)$, $n-l(\lambda^2)<j\le n$, be a positive root with $\la\lambda+\rho,\alpha^\vee\ra-lp>0$ and $\chi(s_{\alpha,l}\cdot\lambda)\ne0$ for some integer $l\ge1$. Then the first $l(\lambda^1)$ entries of $s_{\alpha,l}(\lambda+\rho)$ are distinct and $>n-l(\lambda^1)$ and the last $l(\lambda^2)$ entries are distinct and $\le l(\lambda^2)$.
Now put $\mu=[\mu^1,\mu^2]={\rm sort}(s_{\alpha,l}(\lambda+\rho))-\rho$. Then, $\mu^h$ is a partition with $\mu^h\subsetneqq\lambda^h$ for all $h\in\{1,2\}$,
and $\mu$ is $W_p^{l(\lambda^1),l(\lambda^2)}$-conjugate to $\lambda$ under the dot action. Furthermore, the map $(\alpha,l)\mapsto\mu$ is injective.
\end{lem}

\begin{proof}
The first assertion follows from the last assertion of Lemma~\ref{lem.zero} and the fact that $\chi(s_{\alpha,l}\cdot\lambda)\ne0$. Furthermore, it is also clear that we can sort $s_{\alpha,l}(\lambda+\rho)$ by only permuting the first $l(\lambda^1)$ and the last $l(\lambda^2)$ entries. Since $s_{\alpha,l}(\lambda+\rho)\subseteq \lambda+\rho$ and $\lambda+\rho$ is (strictly) decreasing we will also have ${\rm sort}(s_{\alpha,l}(\lambda+\rho))\subseteq\lambda+\rho$ and therefore $\mu^h$ is a partition with $\mu^h\subsetneqq\lambda^h$ for all $h\in\{1,2\}$. The set of values in $s_{\alpha,l}(\lambda+\rho)$ is obtained by choosing two values in $\lambda+\rho$ and lowering the greatest and increasing the smallest to two new values. So it is clear how to recover $i$, $j$, $a$ and $l$ from the value set of $s_{\alpha,l}(\lambda+\rho)$: $i$ and $j$ are the positions of the two ``old" values of $\lambda+\rho$ that do not occur in $s_{\alpha,l}(\lambda+\rho)$, and $a$ follows from comparing the greatest of the two old values with the greatest of the two new values.
\end{proof}

\begin{exa}
If $\lambda^1$ and $\lambda^2$ are $p$-cores $\Delta(\lambda)$ may have composition factors $L(\mu)$ with $\mu^h\nsubseteq\lambda^h$ for some $h\in\{1,2\}$.
For example, take $p=3$, $n=4$ and $\lambda=[31,1]$. Let $[\xi]$ denote $[\xi,\emptyset]$. Then $\lambda^1=31$ and $\lambda^2=1$ are $p$-cores and we have $JSF([1^3])=0$, $JSF([21])=\chi([1^3])$, $JSF([3])=-\chi([1^3])+\chi([21])={\rm ch}\,L([21])$, $JSF(\lambda)=\chi([21])+\chi([3])={\rm ch}\,L([1^3])+2{\rm ch}\,L([21])+{\rm ch}\,L([3])$. So $L([1^3])$ is a composition factor of $\Delta(\lambda)$ and $1^3\nsubseteq\lambda^1$.
\end{exa}

Note that any partition $\xi$ with $\xi_1+l(\xi)\le p$ is a $p$-core, since $\xi_1+l(\xi)-1$ is the greatest hook length. Eventually, we will need that both $\lambda^1$ and $\lambda^2$ satisfy this stronger assumption.
Define the partial order $\preceq$ on $X^+$ as follows:

\medskip

\emph{$\mu\preceq\lambda$ if and only if there is a sequence of dominant weights $\lambda=\lambda_1,\ldots,\lambda_t=\mu$, $t\ge1$, such that for all $r\in\{1,\ldots,t-1\}$, $\lambda_{r+1}=ws_{\alpha,l}\cdot\lambda_r$ for some $w\in W^{\lambda^1_r,\lambda^2_r}$, $\alpha=\ve_i-\ve_j$, $1\le i\le l(\lambda_r^1)$, $n-l(\lambda_r^2)<j\le n$, and $l\ge1$ with $\la\lambda_r+\rho,\alpha^\vee\ra-lp\ge1$, and $\chi(s_{\alpha,l}\cdot\lambda_r)\ne0$.}

\medskip

Note that when $\lambda^h_1+l(\lambda^h)\le p$ for all $h\in\{1,2\}$, $\mu\preceq\lambda$ implies that $\mu^h\subseteq\lambda^h$ for all $h\in\{1,2\}$ and $\mu$ is $W_p^{l(\lambda^1),l(\lambda^2)}$-conjugate to $\lambda$ under the dot action. This, in turn, implies that $\mu\le\lambda$ and $|\lambda^1|-|\lambda^2|=|\mu^1|-|\mu^2|$. 
Put $$\Lambda_p=\{\mu\in X^+\,|\,\mu^h_1+l(\mu^h)\le p\text{\ for all\ }h\in\{1,2\}\}\,.$$

Assertion (i) below says that, when $\lambda^h_1+l(\lambda^h)\le p$ for some $h\in\{1,2\}$, nonzero contributions of roots $\alpha=\ve_i-\ve_j$, $1\le i\le l(\lambda^1)$, $n-l(\lambda^2)<j\le n$, have a unique $l$-value.

\begin{prop}\label{prop.linkage}
Let $\lambda\in X^+$.
\begin{enumerate}[{\rm(i)}]
\item Assume $\lambda^h_1+l(\lambda^h)\le p$ for some $h\in\{1,2\}$. If $\alpha=\ve_i-\ve_j$, $1\le i\le l(\lambda^1)$, $n-l(\lambda^2)<j\le n$, and $l,a$ are integers $\ge1$ such that $\la\lambda+\rho,\alpha^\vee\ra=a+lp$ and $\chi(s_{\alpha,l}\cdot\lambda)\ne0$, then $a<p$.
\item If $\Lambda\subseteq\Lambda_p$ is $\preceq$-saturated, then the algebra $O_\Lambda(k[G])^*$ is quasihereditary with partially ordered labelling set $(\Lambda,\preceq)$ and the Weyl and induced modules as standard and costandard modules. In particular, if $[\Delta(\lambda):L(\mu)]$ or $(T(\lambda):\nabla(\mu))$ is nonzero, then $\mu\preceq\lambda$.
\item If $\mu$ is maximal with respect to $\preceq$ amongst the dominant weights $\nu$ for which $\chi(\nu)$ occurs in the RHS of the reduced Jantzen Sum Formula associated to $\lambda$ or amongst the dominant weights $\prec\lambda$, then we have \break $\dim\Hom_G(\nabla(\lambda),\nabla(\mu))=[\Delta(\lambda):L(\mu)]\ne0$.
\end{enumerate}
\end{prop}

\begin{proof}
(i). Let $\alpha$, $i$, $j$, $l$ be as stated and assume $\lambda^1_1+l(\lambda^1)\le p$. If $a\ge p$, then $(\lambda+\rho)_i-a\le (\lambda+\rho)_i-p\le\lambda^1_i+n-p\le n-l(\lambda^1)$, which contradicts Lemma~\ref{lem.zero}. If $\lambda^2_1+l(\lambda^2)\le p$ and $a\ge p$, then $(\lambda+\rho)_j+a\ge(\lambda+\rho)_j+p>p-\lambda^2_{j'}\ge l(\lambda^2)$, which contradicts Lemma~\ref{lem.zero}.\\
(ii) and (iii) are proved as in the proof of \cite[Prop~3.1]{LT}, where in (ii) we do the induction on $|\lambda^1|+|\lambda^2|$.
\end{proof}

\section{Translation Functors}\label{s.translation}
The results in this section are analogues of \cite[Thm~3.2,3.3, Prop~3.4]{CdV}, \cite[II.7.9, 7.14-16]{Jan} and \cite[Sect~4]{LT}. Our results don't follow from the ones in \cite{Jan}, see Remark~\ref{rem.translation}. As in \cite[Sect~4]{LT} we will not try to reformulate/generalise these results in terms of $W_p^{s_1,s_2}$ and a type $A_{s_1+s_2-1}$ alcove geometry analogous to \cite[Sect~5-7]{CdVM2} in the Brauer algebra case, 
but we will choose a ``combinatorial" approach similar to \cite{CdV}, using the notion of the ``support" of a dominant weight. This suffices for our applications in Sections~\ref{s.filt_mult} and \ref{s.dec_num}. As in \cite[Sect~4]{LT} the notion of the support of a dominant weight arises from an application of Brauer's formula \cite[II.5.8]{Jan} and the role of the induction and restriction functors in \cite{CdV} is in our setting played by the translation functors.

Recall that the tensor product of two modules with a good/Weyl filtration has a good/Weyl filtration, see \cite[II.4.21, 2.13]{Jan}. Let $\lambda\in X^+$. Then we have by Brauer's formula that $\chi(\lambda)\sum_{i=1}^ne(\ve_i)=\sum_{\mu\in{\rm Supp}_1(\lambda)}\chi(\mu)$ and $\chi(\lambda)\sum_{i=1}^ne(-\ve_i)=\sum_{\mu\in{\rm Supp}_2(\lambda)}\chi(\mu)$, where ${\rm Supp}_1(\lambda)$ consists of all $\mu=[\mu^1,\mu^2]$\\$\in X^+$ which can be obtained by adding a box to $\lambda^1$ or removing a box from $\lambda^2$, but not both, and ${\rm Supp}_2(\lambda)$ consists of all $\mu=[\mu^1,\mu^2]\in X^+$ which can be obtained by removing a box from $\lambda^1$ or adding a box to $\lambda^2$, but not both. Here we used the rules for $\chi(\lambda)$ to be nonzero from Section~\ref{s.JSF}. Note that $\mu\in{\rm Supp}_1(\lambda)$ if and only if $\lambda\in{\rm Supp}_2(\mu)$.
Since ${\rm ch}\,V=\sum_{i=1}^ne(\ve_i)$, it follows that $\nabla(\lambda)\ot V$ resp. $\nabla(\lambda)\ot V^*$ has a good filtration with sections $\nabla(\mu)$, $\mu\in{\rm Supp}_1(\lambda)$ resp. $\mu\in{\rm Supp}_2(\lambda)$. Similarly, since ${\rm ch}\,V^*=\sum_{i=1}^ne(-\ve_i)$, it follows that $\Delta(\lambda)\ot V$ resp. $\Delta(\lambda)\ot V^*$ has a Weyl filtration with sections $\Delta(\mu)$, $\mu\in{\rm Supp}_1(\lambda)$ resp. $\mu\in{\rm Supp}_2(\lambda)$.

We recall the definition and basic properties of the translation functors. For $\lambda\in X^+$ the projection functor ${\rm pr}_\lambda:\{G\text{-modules}\}\to\{G\text{-modules}\}$ is defined by ${\rm pr}_\lambda M=O_{W_p\cdot\lambda\cap X^+}(M)$. Then $M=\bigoplus_\lambda{\rm pr}_\lambda M$ where the sum is over a set of representatives of the linkage classes in $X^+$, see \cite[II.7.3]{Jan}.
Now let $\lambda,\lambda'\in X^+$ with $\lambda'\in{\rm Supp}_h(\lambda)$, $h\in\{1,2\}$. Then $\lambda'-\lambda=\ve_i$ for some $i$ when $h=1$ and $\lambda'-\lambda=-\ve_i$ for some $i$ when $h=2$. We define the \emph{translation functor} $T_\lambda^{\lambda'}:\{G\text{-modules}\}\to\{G\text{-modules}\}$ by $T_\lambda^{\lambda'}M={\rm pr}_{\lambda'}(({\rm pr}_\lambda M)\ot V)$ when $h=1$ and by $T_\lambda^{\lambda'}M={\rm pr}_{\lambda'}(({\rm pr}_\lambda M)\ot V^*)$ when $h=2$.
So this is just a special case of the translation functors from \cite[II.7.6]{Jan}, since the dominant $W$-conjugate of $\lambda'-\lambda$ is $\ve_1$ ($h=1$) or $-\ve_n$ ($h=2$), and $V=\nabla(\ve_1)=L(\ve_1)$ and $V^*=\nabla(-\ve_n)=L(-\ve_n)$.
In particular, $T_\lambda^{\lambda'}$ is exact and left and right adjoint to $T_{\lambda'}^\lambda$.
Note that, for $\lambda'\in{\rm Supp}_h(\lambda)$, $h\in\{1,2\}$, and $\mu\in X^+\cap  W_p\cdot\lambda$, $T_\lambda^{\lambda'}\nabla(\mu)$ has a good filtration with sections $\nabla(\nu)$, $\nu\in{\rm Supp}_h(\mu)\cap  W_p\cdot\lambda'$, and the analogue for Weyl modules and Weyl filtrations also holds.

Recall the definition of the set $\Lambda_p$ from Section~\ref{s.JSF}.
\begin{prop}[Translation equivalence]\label{prop.trans_equivalence}
Let $h,\ov h\in\{1,2\}$ be distinct, let $\lambda,\lambda'\in\Lambda_p$ with $\lambda'\in{\rm Supp}_h(\lambda)$ and let $\Lambda\subseteq W_p\cdot\lambda\cap\Lambda_p,\Lambda'\subseteq W_p\cdot\lambda'\cap\Lambda_p$ be $\preceq$-saturated sets. 
Assume
\begin{enumerate}[{\rm (1)}]
\item ${\rm Supp}_h(\nu)\cap W_p\cdot\lambda'\subseteq\Lambda_p$ for all $\nu\in\Lambda$, and ${\rm Supp}_{\ov h}(\nu')\cap  W_p\cdot\lambda\subseteq\Lambda_p$ for all $\nu'\in\Lambda'$.
\item $|{\rm Supp}_h(\nu)\cap W_p\cdot\lambda'|=1=|{\rm Supp}_{\ov h}(\nu')\cap W_p\cdot\lambda|$ for all $\nu\in\Lambda$ and $\nu'\in\Lambda'$.
\item The map $\nu\mapsto\nu':\Lambda\to\Lambda_p$ given by ${\rm Supp}_h(\nu)\cap W_p\cdot\lambda'=\{\nu'\}$ has image $\Lambda'$, and together with its inverse $\Lambda'\to\Lambda$ it preserves the order $\preceq$.
\end{enumerate}
Then $T_\lambda^{\lambda'}$ restricts to an equivalence of categories $\mc C_\Lambda\to\mc C_{\Lambda'}$ with inverse $T_{\lambda'}^\lambda:\mc C_{\Lambda'}\to\mc C_\Lambda$.
Furthermore, with $\nu$ and $\nu'$ as in (3), we have $T_\lambda^{\lambda'}\nabla(\nu)=\nabla(\nu')$, $T_\lambda^{\lambda'}\Delta(\nu)=\Delta(\nu')$, $T_\lambda^{\lambda'}L(\nu)=L(\nu')$, $T_\lambda^{\lambda'}T(\nu)=T(\nu')$ and $T_\lambda^{\lambda'}I_\Lambda(\nu)=I_{\Lambda'}(\nu')$.
\end{prop}

\begin{proof}
The proof is a straightforward simplification of the proof of \cite[Prop~4.1]{LT}: We can work with ordinary instead of refined translation functors. We give it here for convenience of the reader.
The first assertion and the identities involving the induced and Weyl modules are obvious. We have an exact sequence
\begin{align}\label{eq.Delta}0\to M\to\Delta(\nu)\to L(\nu)\to 0\,,\end{align}
where all composition factors $L(\eta)$ of $M$ satisfy $\eta\prec\nu$. Applying $T_\lambda^{\lambda'}$ gives the exact sequence
\begin{align}\label{eq.T_Delta}0\to T_\lambda^{\lambda'}M\to\Delta(\nu')\to T_\lambda^{\lambda'}L(\nu)\to 0\,.\end{align}
Using the order preserving properties of $\nu\mapsto\nu'$ we see that for any $\theta\in\Lambda$ all composition factors $L(\eta')$ of $T_\lambda^{\lambda'}L(\theta)$ satisfy $\eta'\preceq\theta'$. So all composition factors $L(\eta')$ of $T_\lambda^{\lambda'}M$ satisfy $\eta'\prec\nu'$.
Therefore $T_\lambda^{\lambda'}L(\nu)$ must have simple head $L(\nu')$ and all other composition factors $L(\eta')$ satisfy $\eta'\prec\nu'$. If $T_\lambda^{\lambda'}L(\nu)\ne L(\nu')$, then $$\Hom_G(\Delta(\eta),L(\nu))=\Hom_G(T_{\lambda'}^\lambda\Delta(\eta'),L(\nu))=\Hom_G(\Delta(\eta'), T_\lambda^{\lambda'}L(\nu))\ne0$$
for some $\eta\ne\nu$. This is clearly impossible, so $T_\lambda^{\lambda'}L(\nu)=L(\nu')$. We can prove the same for $T_{\lambda'}^\lambda$, and then we can deduce as in the proof \cite[II.7.9]{Jan} that $T_{\lambda'}^\lambda T_\lambda^{\lambda'}\cong\id_{\mc C_\Lambda}$ and $T_\lambda^{\lambda'}T_{\lambda'}^\lambda\cong\id_{\mc C_{\Lambda'}}$. This implies the remaining assertions.
\end{proof}

\begin{prop}[Translation projection]\label{prop.trans_projection}
Let $h,\ov h\in\{1,2\}$ be distinct, let $\lambda,\lambda'\in\Lambda_p$ with $\lambda'\in{\rm Supp}_h(\lambda)$ and let $\Lambda\subseteq W_p\cdot\lambda\cap\Lambda_p,\Lambda'\subseteq W_p\cdot\lambda'\cap\Lambda_p$ be $\preceq$-saturated sets. Put $\tilde\Lambda=\{\nu\in\Lambda\,|\,{\rm Supp}_h(\nu)\cap W_p\cdot\lambda'\ne\emptyset\}$. Assume
\begin{enumerate}[{\rm (1)}]
\item ${\rm Supp}_h(\nu)\cap W_p\cdot\lambda'\subseteq\Lambda_p$ for all $\nu\in\Lambda$, and ${\rm Supp}_{\ov h}(\nu')\cap W_p\cdot\lambda\subseteq\Lambda_p$ for all $\nu'\in\Lambda'$.
\item $|{\rm Supp}_h(\nu)\cap W_p\cdot\lambda'|=1$ for all $\nu\in\tilde\Lambda$, and $|{\rm Supp}_{\ov h}(\nu')\cap W_p\cdot\lambda|=2$ for all $\nu'\in\Lambda'$.
\item The map $\nu\mapsto\nu':\tilde\Lambda\to\Lambda_p$ given by ${\rm Supp}_h(\nu)\cap W_p\cdot\lambda'=\{\nu'\}$ is a 2-to-1 map which has image $\Lambda'$ and preserves the order $\preceq$. For $\nu'\in\Lambda'$ we can write ${\rm Supp}_{\ov h}(\nu')\cap W_p\cdot\lambda=\{\nu^+,\nu^-\}$ with $\nu^-\prec\nu^+$ and then $\Hom_G(\nabla(\nu^+),\nabla(\nu^-))\ne0$ and $\eta'\preceq\nu'\Rightarrow\eta^+\preceq\nu^+$ and $\eta^-\preceq\nu^-$.
\end{enumerate}
Then $T_\lambda^{\lambda'}$ restricts to a functor $\mc C_\Lambda\to\mc C_{\Lambda'}$ and $T_{\lambda'}^\lambda$ restricts to a functor $\mc C_{\Lambda'}\to\mc C_\Lambda$.
Now let $\nu\in\Lambda$. If $\nu\notin\tilde\Lambda$, then $T_\lambda^{\lambda'}\nabla(\nu)= T_\lambda^{\lambda'}\Delta(\nu)= T_\lambda^{\lambda'}L(\nu)=0$.
For $\nu'\in\Lambda'$ with $\nu^\pm$ as in (3), we have $T_\lambda^{\lambda'}\nabla(\nu^\pm)=\nabla(\nu')$, $T_\lambda^{\lambda'}\Delta(\nu^\pm)=\Delta(\nu')$, $T_\lambda^{\lambda'}L(\nu^-)=L(\nu')$, $T_\lambda^{\lambda'}L(\nu^+)=0$, $T_{\lambda'}^\lambda T(\nu')=T(\nu^+)$ and $T_{\lambda'}^\lambda I_{\Lambda'}(\nu')=I_\Lambda(\nu^-)$.
\end{prop}
\begin{proof}
Again, the proof is a straightforward simplification of the proof of \cite[Prop~4.2]{LT}. We give it here for convenience of the reader.
The identities involving the induced and Weyl modules are obvious. Moreover, it is also clear that $T_\lambda^{\lambda'}L(\nu)=0$ when $\nu\notin\tilde\Lambda$, since $T_\lambda^{\lambda'}\Delta(\nu)$ surjects onto $T_\lambda^{\lambda'}L(\nu)$ and is $0$. If $\eta\prec\nu^-$, then $\eta'\prec\nu'$, 
so we obtain $T_\lambda^{\lambda'}L(\nu^-)=L(\nu')$ as in the proof of Proposition~\ref{prop.trans_equivalence}.
Now consider \eqref{eq.Delta} and \eqref{eq.T_Delta} for $\nu=\nu^+$. Since $[\Delta(\nu^+):L(\nu^-)]=[\nabla(\nu^+):L(\nu^-)]\ne0$, we know that $L(\nu^-)$ occurs in $M$. So $T_\lambda^{\lambda'}L(\nu^-)=L(\nu')$ occurs in $T_\lambda^{\lambda'}M$ and therefore not in $T_\lambda^{\lambda'}L(\nu^+)$.
If $T_\lambda^{\lambda'}L(\nu^+)\ne0$, then it would have simple head $L(\nu')$ by \eqref{eq.T_Delta}. So $T_\lambda^{\lambda'}L(\nu^+)=0$.
Note that ${\rm ch}\, T_\lambda^{\lambda'} T_{\lambda'}^\lambda M=2{\rm ch}\,M$ for any $M\in\mc C_{\Lambda'}$ which has a good or Weyl filtration. Now the equality $T_{\lambda'}^\lambda T(\nu')=T(\nu^+)$ is proved as in \cite[E.11]{Jan}, replacing $\uparrow, w\cdot\lambda, ws\cdot\lambda, w\cdot\mu, T_\lambda^\mu$ and $T_\mu^\lambda$ by $\preceq, \nu^+, \nu^-, \mu,  T_\lambda^{\lambda'}$ and $T_{\lambda'}^\lambda$. Finally,
\begin{align}\label{eq.I}\Hom_{\mc C_\Lambda}(-, T_{\lambda'}^\lambda I_{\Lambda'}(\nu'))=\Hom_{\mc C_{\Lambda'}}(-,I_{\Lambda'}(\nu'))\circ  T_\lambda^{\lambda'}\end{align}
is exact, so $T_{\lambda'}^\lambda I_{\Lambda'}(\nu')$ is injective in $\mc C_\Lambda$. Applying both sides of \eqref{eq.I} to $L(\eta)$, for $\eta\notin\tilde\Lambda$, for $\eta=\eta^+$ and for $\eta=\eta^-$, shows that $T_{\lambda'}^\lambda I_{\Lambda'}(\nu')$ has simple socle $L(\nu^-)$ and therefore equals $I_\Lambda(\nu^-)$.
\end{proof}

\begin{rem}\label{rem.translation}
The translated weight $\lambda'$ need not be in the facet closure of $\lambda$. For example, when $p=5$, $n=5$, $s_1=s_2=1$ and $(\lambda,\lambda')=([4,4],[3,4])$ or $([3,4],[2,4])$, then it is easy to find affine reflection hyperplanes which contain $\lambda$, but not $\lambda'$. 
However, we can, for $\Lambda=\{[4,4],[2,2]\}$ and $\Lambda'=\{[3,4],[2,3]\}$, apply Proposition~\ref{prop.trans_equivalence} in the first case, and, for $\Lambda=\{[3,4],[2,3]\}$ and $\Lambda'=\{[2,4]\}$, apply Proposition~\ref{prop.trans_projection} in the second case. We refer to Section~\ref{s.arrow_diagrams} for how to express this in terms of arrow diagrams. 
\end{rem}

\section{Arrow diagrams}\label{s.arrow_diagrams}
This section is based on the approaches of \cite{CdV} and \cite{Sh}. We use ``characteristic $p$ walls" as in \cite{Sh}. Recall the definition of $\rho$ from Section~\ref{s.prelim} and recall from Section~\ref{s.JSF} that $i'=n+1-i$. 
An arrow diagram has $p$ \emph{nodes} on a (horizontal) \emph{line} with $p$ \emph{labels}: $0,\ldots,p-1$. The $i$-th node from the left has label $i-1$.
Although $0$ and $p-1$ are not connected we consider them as neighbours and we will identify a diagram with any of its cyclic shifts. So when we are going to the left through the nodes we get $p-1$ after $0$ and when we are going to the right we get $0$ after $p-1$.
Next we choose $s_1,s_2\in\{1,\ldots,\min(n,p)\}$ with $s_1+s_2\le n$ and put a wall below the line between $\rho_{s_1}$ and $\rho_{s_1}-1$ mod $p$, and a wall above the line between $\rho_{s_2'}=s_2$ and $s_2+1$ mod $p$. Then we can also put in a top and bottom \emph{value} for each label. A value and its corresponding label are always equal mod $p$. Below the line we start with $\rho_{s_1}$ immediately to the right of the wall, and then increasing in steps of $1$ going to the right: $\rho_{s_1},\rho_{s_1}+1,\ldots,\rho_{s_1}+p-1$. Above the line we start with $\rho_{s_2'}=s_2$ immediately to the left of the wall, and then decreasing in steps of $1$ going to the left: $s_2,s_2-1,\ldots,s_2-p+1$.
For example, when $p=5$, $n=5$ and $s_1=s_2=1$, then $\rho_{s_1}=s_1'=5$, $\rho_{s_2'}=s_2=1$ and we have labels
$$
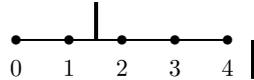
\resizebox{3.2cm}{.5cm}{\xy
(0,0)="a1"*{\bullet};(0,-4.5)*{0};
(8.5,0)="a2"*{\bullet};(8.5,-4.5)*{1};(12.8,3)*{\rule[0mm]{.3mm}{6mm}};
(17,0)="a3"*{\bullet};(17,-4.5)*{2};
(25.5,0)="a4"*{\bullet};(25.5,-4.5)*{3};
(34,0)="a5"*{\bullet};(34,-4.5)*{4};(38,-3)*{\rule[0mm]{.3mm}{6mm}};
{"a1";"a5"**@{-}}; 
\endxy}\ \ \,
$$
and values
$$
\resizebox{3.2cm}{.5cm}
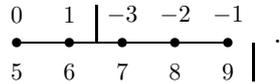
{\xy
(0,0)="a1"*{\bullet};(0,4.5)*{0};(0,-4.5)*{5};
(8.5,0)="a2"*{\bullet};(8.5,4.5)*{1};(8.5,-4.5)*{6};(12.8,3)*{\rule[0mm]{.3mm}{6mm}};
(17,0)="a3"*{\bullet};(17,4.5)*{-3};(17,-4.5)*{7};
(25.5,0)="a4"*{\bullet};(25.5,4.5)*{-2};(25.5,-4.5)*{8};
(34,0)="a5"*{\bullet};(34,4.5)*{-1};(34,-4.5)*{9};(38,-3)*{\rule[0mm]{.3mm}{6mm}};
{"a1";"a5"**@{-}}; 
\endxy}\ \ .
$$
For a dominant weight $\lambda=[\lambda^1,\lambda^2]$ with $l(\lambda^h)\le s_h\le p-\lambda^h_1$ for all $h\in\{1,2\}$ we now form the ($(s_1,s_2)$-)\emph{arrow diagram} by putting in $s_1$ arrows below the line ($\land$) that point \emph{from} the values $(\rho+\lambda)_1,\ldots,(\rho+\lambda)_{s_1}$, or from the corresponding labels, and $s_2$ arrows above the line ($\vee$) that point from the values $(\rho+\lambda)_{1'},\ldots,(\rho+\lambda)_{s_2'}$, or to the corresponding labels.
So in the above example the arrow diagram of $\lambda=[4,4]$ is
$$
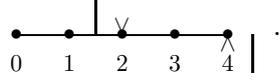
\resizebox{3.2cm}{.5cm}{\xy
(0,0)="a1"*{\bullet};(0,-4.5)*{0};
(8.5,0)="a2"*{\bullet};(8.5,-4.5)*{1};(12.8,3)*{\rule[0mm]{.3mm}{6mm}};
(17,0)="a3"*{\bullet};(17,-4.5)*{2};(17,1.5)*{\vee};
(25.5,0)="a4"*{\bullet};(25.5,-4.5)*{3};
(34,0)="a5"*{\bullet};(34,-4.5)*{4};(34,-1.5)*{\land};(38,-3)*{\rule[0mm]{.3mm}{6mm}};
{"a1";"a5"**@{-}}; 
\endxy}\ \ .
$$
In such a diagram we frequently omit the nodes and/or the labels. When it has already been made clear what the labels are and where the walls are, we can simply represent the arrow diagram by a string of single arrows ($\land$, $\vee$), opposite pairs of arrows ($\times$) and symbols ${\rm o}$ to indicate the absence of an arrow. In the above example $\lambda=[4,4]$ is then represented by ${\rm oo}\nts\vee\nts{\rm o}\land$ and $\lambda=[2,4]$ is represented by ${\rm oo}\nts\times\nts{\rm oo}$.

We can form the arrow diagram of $\lambda$ as follows. First line up $s_1$ arrows immediately to the right of the wall below the line and then move them to the right to the correct positions. The arrow furthest from the wall corresponds to $\lambda^1_1$, and the arrow closest to the wall corresponds to $\lambda^1_{s_1}$. Then line up $s_2$ arrows immediately to the left of the wall above the line and then move them to the left to the correct positions. The arrow furthest from the wall corresponds to $\lambda^2_1$, and the arrow closest to the wall corresponds to $\lambda^2_{s_2}$.

The part of $\lambda^1$ corresponding to an arrow below the line equals the number of nodes without a $\land$ from that arrow to the wall going to the left.
From the diagram you can see what you can do with the wall below the line, changing $s_1$ but not $\lambda$:
If there is an arrow immediately to the right of the wall, i.e. $l(\lambda^1)<s_1$, then you can move that wall one step to the right, removing the arrow that you move it past.
If there is no arrow immediately to the left of the wall, i.e. $\lambda^1_1<p-s_1$, then you can move the wall one step to the left, putting a $\land$ at the node that you move it past, provided $s_1<n-s_2$. The analogous assertions for the wall above the line are obtained by replacing ``right", ``left", $\lambda^1$, $\land$, $s_1$ and $s_2$ by ``left", ``right", $\lambda^2$, $\vee$, $s_2$ and $s_1$.

More generally, we can for any $s_1,s_2\in\{1,\ldots,n\}$ with $s_1+s_2\le n$ and $\mu\in X^+$ with $l(\mu^h)\le s_h$ for all $h\in\{1,2\}$, put $s_1$ arrows below the line in the diagram pointing from the labels equal to $(\rho+\mu)_1,\ldots,(\rho+\mu)_{s_1}$ mod $p$, and $s_2$ arrows above the line in the diagram pointing to the labels equal to $(\rho+\mu)_{1'},\ldots,(\rho+\mu)_{s_2'}$ mod $p$, allowing repeated $\vee$'s or $\land$'s at a node. Then $\mu$ and $\nu$ with $l(\mu^h),l(\nu^h)\le s_h$ for all $h\in\{1,2\}$ are $W_p$-conjugate under the dot action if and only if $|\mu|=|\nu|$ and the arrow diagrams of $\mu$ and $\nu$ have the same number of arrows at each node, if and only if $|\mu|=|\nu|$ and the arrow diagram of $\nu$ can be obtained from that of $\mu$ by choosing a certain number of $\land$'s and an equal number of $\vee$'s and replacing all these arrows by their opposites.

From now on $s_1,s_2\in\{1,\ldots,\min(n,p)\}$ with $s_1+s_2\le n$, unless stated otherwise. We put
$$\Lambda(s_1,s_2)=\{\lambda\in X^+\,|\,l(\lambda^h)\le s_h\le p-\lambda^h_1\text{\ for all\ }h\in\{1,2\}\}\,.$$
Unless stated otherwise, we assume $\lambda\in\Lambda(s_1,s_2)$.

When we speak of ``arrow pairs" it is understood that both arrows are single, i.e. neither of the two arrows is part of an $\times$. So, for example, at the node of the first arrow in an arrow pair $\vee\land$ there should not also be a $\land$. The arrows need not be consecutive in the diagram.

We now define the \emph{cap diagram} $c_\lambda$ of the arrow diagram associated to $\lambda$ as follows. We assume that the arrow diagram is cyclically shifted such that at least one of the walls is between the first and last node. We select one such wall and when we speak of ``the wall" it will be the other wall. All caps are anti-clockwise, starting from the rightmost node. We start on the left side of the wall. We form the caps recursively. Find an arrow pair $\vee\land$ that are neighbours in the sense that the only arrows in between are already connected with a cap or are part of an $\times$, and connect them with a cap. Repeat this until there are no more such arrow pairs. Now the unconnected arrows that are not part of an $\times$ form a sequence $\land\cdots\land\vee\cdots\vee$. Note that none of these arrows occur inside a cap. The caps on the right side of the wall are formed in the same way.
For example, when $p=17$, $n=20$, $s_1=8$, $s_2=7$ and $\lambda=[9654^22,8^243^22]$, then
$c_\lambda$ is
$$\xy
(0,0)="a1";(0,-1)*{\land};(0,-4)*{13};
(5,0)="a2";(5,-1)*{\land};
(10,0)="a3";
(15,0)="a4";(15,1)*{\vee};(15,-4)*{16};
(20,0)="a5";(20,-1)*{\land};(20,-4)*{0};
(25,0)="a6";(25,1)*{\vee};
(30,0)="a7";(30,1)*{\vee};
(35,0)="a8";(35,-1)*{\land};
(40,0)="a9";(40,-1)*{\land};(40,1)*{\vee};
(45,0)="a10";
(50,0)="a11";(50,-1)*{\land};
(55,0)="a12";(55,1)*{\vee};(57.5,3)*{\rule[0mm]{.3mm}{6mm}};
(60,0)="a13";(60,-1)*{\land};
(65,0)="a14";
(70,0)="a15";(70,1)*{\vee};
(75,0)="a16";(75,1)*{\vee};
(80,0)="a17";(80,-1)*{\land};(80,-4)*{12};(82.5,-3)*{\rule[0mm]{.3mm}{6mm}};
{"a1";"a17"**@{-}}; 
"a5";"a4"**\crv{(20,6)&(15,6)};
"a8";"a7"**\crv{(35,6)&(30,6)};
"a11";"a6"**\crv{(50,8)&(25,11)};
"a17";"a16"**\crv{(80,6)&(75,6)};
\endxy\ .$$
Note that the nodes with labels $5,9,15$ have no arrow.

\begin{lem}\label{lem.JSF-arrows}
Let $\lambda\in\Lambda(s_1,s_2)$. Assume that the arrow diagram of $\lambda$ is cyclically shifted such that at least one of the walls is between the first and last node.
\begin{enumerate}[{\rm(i)}]
\item The nonzero terms in the reduced Jantzen Sum Formula associated to $\lambda$ correspond in the arrow diagram of $\lambda$ to the arrow pairs $\vee\land$ to the left or to the right of the wall.
\item $\Delta(\lambda)$ is irreducible (equivalently, $T(\lambda)=\Delta(\lambda)$ or $\nabla(\lambda)$) if and only if there are no caps in $c_\lambda$.
\item If $\mu$ is obtained from $\lambda$ by reversing the arrows in a pair as in (i) with consecutive arrows (no single arrows in between), then $\dim\Hom_G(\nabla(\lambda),\nabla(\mu))$ $=[\Delta(\lambda):L(\mu)]\ne0$.
\end{enumerate}
\end{lem}

\begin{proof}
We will work with the ``unshifted"  diagram, so the leftmost node has label $0$.
When $s_2=p$, then there are no single $\land$'s and $\lambda^2=0$, so the reduced sum is empty and the assertion is trivially true. So we assume $\rho_{s_2'}=s_2<p$.
Write $\rho_{s_1}=x_1+up$ with $0\le x_1<p$, $u\ge0$. The general form of a value above the line is
$
\begin{smallmatrix}
x&\rule[-.7mm]{.3mm}{3mm}&x-p\\
\hline
\end{smallmatrix}
$
and below the line it is
$
\begin{smallmatrix}
\hline\\
x+(u+1)p&\rule[-.7mm]{.3mm}{3mm}&x+up
\end{smallmatrix}
$.
Here $x$ always satisfies $0\le x<p$. Note that the ``opposite" value on the other side of the line has the same $x$ in its general form.
Put differently, the label corresponding to the value is $x$.
Now let $\alpha=\ve_i-\ve_j$, $1\le i\le l(\lambda^1)$, $n-l(\lambda^2)<j\le n$, and $l,a\ge1$ such that $\la\lambda+\rho,\alpha^\vee\ra=a+lp$ and $\chi(s_{\alpha,l}\cdot\lambda)\ne0$.
Put $c=(\lambda+\rho)_i$ and $d=(\lambda+\rho)_j$. Note that $c\ne d$ mod $p$, because otherwise we would have $a=0$. Assume that the wall above the line is to the left of or above the wall below the line ($x_1>s_2$). Then the 12 candidate configurations of $c$ and $d$ in the arrow diagram of $\lambda$ are:
$$
%
%
\begin{smallmatrix}
d\ &\rule{.3mm}{2mm}\ \ &\vspace{.7mm}\\
\hline\\
\ c\nts&\ \ \rule{.3mm}{2mm}&\ 
\end{smallmatrix}\,,
\begin{smallmatrix}
\ d\nts&\rule{.3mm}{2mm}\ \ &\vspace{.7mm}\\
\hline\\
c\ &\ \ \rule{.3mm}{2mm}&\ 
\end{smallmatrix}\,,
\begin{smallmatrix}
d&\rule{.3mm}{2mm}\ \ &\vspace{.7mm}\\
\hline\\
&\ c\ \nts\rule{.3mm}{2mm}&\ 
\end{smallmatrix}\,,
\begin{smallmatrix}
d&\rule{.3mm}{2mm}\ \ &\vspace{.7mm}\\
\hline\\
&\ \ \rule{.3mm}{2mm}&c
\end{smallmatrix}\,,
%
%
\begin{smallmatrix}
\ &\rule{.3mm}{2mm}\,d\ &\vspace{.7mm}\\
\hline\\
c&\ \ \rule{.3mm}{2mm}&\ 
\end{smallmatrix}\,,
\begin{smallmatrix}
\ &\rule{.3mm}{2mm}\,d\ \ &\vspace{.7mm}\\
\hline\\
&\ \ c\,\rule{.3mm}{2mm}&\ 
\end{smallmatrix}\,,
\begin{smallmatrix}
\ &\rule{.3mm}{2mm}\ \ d\,&\vspace{.7mm}\\
\hline\\
&\,c\ \ \rule{.3mm}{2mm}&\ 
\end{smallmatrix}\,,
\begin{smallmatrix}
\ &\rule{.3mm}{2mm}\,d\ &\vspace{.7mm}\\
\hline\\
&\ \ \rule{.3mm}{2mm}&\nts c
\end{smallmatrix}\,,$$
$$
%
%
\begin{smallmatrix}
\ &\rule{.3mm}{2mm}\ \ &\Nts d\vspace{.7mm}\\
\hline\\
c&\ \ \rule{.3mm}{2mm}&\ 
\end{smallmatrix}\,,
\begin{smallmatrix}
\ &\rule{.3mm}{2mm}\ \ &\Nts d\vspace{.7mm}\\
\hline\\
&\ c\ \rule{.3mm}{2mm}&\ 
\end{smallmatrix}\,,
\begin{smallmatrix}
\ &\rule{.3mm}{2mm}\ \ &\Nts\Nts d\vspace{.7mm}\\
\hline\\
&\ \ \rule{.3mm}{2mm}&\ c
\end{smallmatrix}\,,
\begin{smallmatrix}
\ &\rule{.3mm}{2mm}\ \ &\ d\vspace{.7mm}\\
\hline\\
&\ \ \rule{.3mm}{2mm}&\Nts\Nts c
\end{smallmatrix}\,.
$$
Here it is understood that the opposite values of $c$ and $d$ are not present in the diagram of $\lambda+\rho$, since otherwise $s_{\alpha,l}(\lambda+\rho)$ would contain a repeat and $\chi(s_{\alpha,l}\cdot\lambda)$ would be $0$.
Now it is easy to see that the only possible configurations are 1,6,9 and 11:
$
\begin{smallmatrix}
d\ &\rule{.3mm}{2mm}\ \ &\vspace{.7mm}\\
\hline\\
\ c\nts&\ \ \rule{.3mm}{2mm}&\ 
\end{smallmatrix}\,,
\begin{smallmatrix}
\ &\rule{.3mm}{2mm}\,d\ \ &\vspace{.7mm}\\
\hline\\
&\ \ c\,\rule{.3mm}{2mm}&\ 
\end{smallmatrix}\,,
\begin{smallmatrix}
\ &\rule{.3mm}{2mm}\ \ &\Nts d\vspace{.7mm}\\
\hline\\
c&\ \ \rule{.3mm}{2mm}&\ 
\end{smallmatrix}\,,
\begin{smallmatrix}
\ &\rule{.3mm}{2mm}\ \ &\Nts\Nts d\vspace{.7mm}\\
\hline\\
&\ \ \rule{.3mm}{2mm}&\ c
\end{smallmatrix}\,,
$
which correspond precisely to the arrow pairs from the assertion.
For example, for configuration 1 we have $c=x+(u+1)p, d=y$ with $0\le y<x<p$. So $a=x-y$, $l=u+1$, and $s_{\alpha,l}(\lambda+\rho)$ equals $y+(u+1)p$ in position $i$ and $x$ in position $j$. Since these are the available values for the labels $y,x$, this configuration is possible.
Next, for configuration 2 we have $c=x+(u+1)p, d=y$ with $0\le x<y<p$. So $a=p-(y-x)$, $l=u$, and $s_{\alpha,l}(\lambda+\rho)$ equals $y+up$ in position $i$ and $x+p$ in position $j$. However, the available values for the labels $y,x$ are $y+(u+1)p$ and $x$. So this configuration is not possible.
As a final example, for configuration 9 we have $c=x+(u+1)p, d=y-p$ with $0\le y<x<p$. So $a=p-(y-x)$, $l=u+1$, and $s_{\alpha,l}(\lambda+\rho)$ equals $y+up$ in position $i$ and $x$ in position $j$. Since these are the available values for the labels $y,x$, this configuration is possible.
The case that the wall above the line is to the right of or above the wall below the line ($x_1\le s_2+1$) is completely analogous.

Conversely, it is clear that if $(\alpha,l)$ corresponds to one of the stated arrow pairs, then the first $l(\lambda^1)$ entries of $s_{\alpha,l}(\lambda+\rho)$ are distinct and $>n-l(\lambda^1)$ and the last $l(\lambda^2)$ entries are distinct and $\le l(\lambda^2)$, so $\chi(s_{\alpha,l}\cdot\lambda)\ne0$.\\
(ii). This follows easily from (i): there is an arrow pair $\vee\land$ to the left of the wall if and only if there is a cap to the left of the wall in $c_\lambda$ (although there will in general be more such pairs than such caps).\\
(iii). Such a $\mu$ is maximal amongst the weights $\nu$ for which (a nonzero multiple of) $\chi(\nu)$ occurs on the RHS of the reduced Jantzen Sum Formula, so this follows from Proposition~\ref{prop.linkage}(iii).
\end{proof}

\begin{rems}\label{rems.preceq}
1.\ Let $s_1,s_2\in\{1,\ldots,\min(n,p)\}$ with $s_1+s_2\le n$ and let $\lambda\in\Lambda(s_1,s_2)$ and $\mu\in X^+$.
Assume that the nodes are cyclically shifted such that at least one of the walls determined by $s_1$ and $s_2$ is between the first and last node.
Then it follows from the above lemma that $\mu\preceq\lambda$ if and only if $\mu\in\Lambda(s_1,s_2)$ and the arrow diagram of $\mu$ can be obtained from that of $\lambda$ by repeatedly replacing an arrow pair $\vee\land$ to the left or to the right of the wall, by the opposite arrow pair.

Furthermore, $\lambda,\mu\in\Lambda(s_1,s_2)$ are conjugate under the dot action of $W_p$ if and only if the arrow diagram of $\mu$ is obtained from that of $\lambda$
by choosing a certain number of (single) $\land$'s and an equal number of $\vee$'s to the left of the wall and choosing a certain number of $\land$'s and an equal number of $\vee$'s to the right of the wall and then replacing all these arrows by their opposites. This follows by combining our earlier characterisation of $W_p$-conjugacy under the dot action with a computation of the change in coordinate sum in terms of the number of arrows of each general form from the proof of the above lemma.\\ 
2.\ The $l$-values corresponding to the configurations 1,6,9 and 11 from the proof are $u+1,u+2,u+1,u+1$.
The possible configurations when the wall above the line is to the right of or above the wall below the line are:
$
\begin{smallmatrix}
d\ &\ \ \rule{.3mm}{2mm}&\vspace{.7mm}\\
\hline\\
\ c\nts&\rule{.3mm}{2mm}\ \ &\ 
\end{smallmatrix}\,,
\begin{smallmatrix}
\ &\,d\ \ \rule{.3mm}{2mm}&\vspace{.7mm}\\
\hline\\
&\rule{.3mm}{2mm}\ \ c\,&\ 
\end{smallmatrix}\,,
\begin{smallmatrix}
\ &\ \ \rule{.3mm}{2mm}&\nts d\vspace{.7mm}\\
\hline\\
c&\rule{.3mm}{2mm}\ \ &\ 
\end{smallmatrix}
$,
$
\begin{smallmatrix}
\ &\ \ \rule{.3mm}{2mm}&\Nts\nts d\vspace{.7mm}\\
\hline\\
&\rule{.3mm}{2mm}\ \ &\ c
\end{smallmatrix}\,,
$
with $l$-values $u+1,u,u+1,u+1$.
So in the reduced Jantzen Sum Formula associated to $\lambda$ we only have two possible $l$-values.
\end{rems}

\section{Weyl filtration multiplicities in tilting modules}\label{s.filt_mult}
Let $s_1,s_2\in\{1,\ldots,\min(n,p)\}$ with $s_1+s_2\le n$. Recall the definition of the set $\Lambda(s_1,s_2)$ from Section~\ref{s.arrow_diagrams}.
Assume that the nodes are cyclically shifted such that at least one of the walls determined by $s_1$ and $s_2$ is between the first and last node.
Recall that we fix one such wall and that ``the wall" will always refer to the other wall.
Let $\lambda\in\Lambda(s_1,s_2)$, and let $\mu\in X^+$ with $\mu\preceq\lambda$. Then the arrow diagram of $\mu$ has its single arrows and its $\times$'s at the same nodes as the arrow diagram of $\lambda$.
We know, by Remark~\ref{rems.preceq}.1, that the arrow diagram of $\mu$ can be obtained from that of $\lambda$ by repeatedly replacing an arrow pair $\vee\land$ to the left or to the right of the wall by the opposite arrow pair.

Recall the definition of the cap diagram $c_\lambda$ from the previous section. We now define the \emph{cap diagram} $c_{\lambda\mu}$ associated to $\lambda$ \emph{and $\mu$} by replacing each arrow in $c_\lambda$ by the arrow from the arrow diagram of $\mu$ at the same node. Put differently, we put the caps from $c_\lambda$ on top of the arrow diagram of $\mu$. We say that $c_{\lambda\mu}$ is \emph{oriented} if all caps in $c_{\lambda\mu}$ are oriented (clockwise or anti-clockwise). It is not hard to show that when $c_{\lambda\mu}$ is oriented, the arrow diagrams of $\lambda$ and $\mu$ are the same at the nodes which are not endpoints of a cap in $c_\lambda$.
For example, when $p=5$, $n=7$, $s_1=2$, $s_2=3$ and $\lambda=[32,21^2]$. Then $\rho_{s_1}=s_1'=6$, and $c_\lambda$ (shifted) is
$$
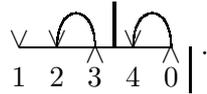
\xy
(0,0)="a1";(0,1)*{\vee};(0,-4)*{1};
(5,0)="a2";(5,1)*{\vee};(5,-4)*{2};
(10,0)="a3";(10,-1)*{\land};(10,-4)*{3};(12.5,3)*{\rule[0mm]{.3mm}{6mm}};
(15,0)="a4";(15,1)*{\vee};(15,-4)*{4};
(20,0)="a5";(20,-1)*{\land};(20,-4)*{0};(22.5,-3)*{\rule[0mm]{.3mm}{6mm}};
{"a1";"a5"**@{-}}; 
"a3";"a2"**\crv{(10,6)&(5,6)}; 
"a5";"a4"**\crv{(20,6)&(15,6)} 
\endxy\ .$$
The $\mu\in X^+$ with $\mu\prec\lambda$ are $[2^2,1^3],[31,21],[21,1^2],[3,2],[2,1]$,
with (shifted) arrow diagrams $\vee\vee\land\land\vee, \vee\land\vee\vee\land, \vee\land\vee\land\vee, \land\vee\vee\vee\land, \land\vee\vee\land\vee$.
Only for the first three $c_{\lambda\mu}$ is oriented. For the first two of these $c_{\lambda\mu}$ has one clockwise cap and for the third both caps are clockwise.

\begin{thm}\label{thm.filt_mult}
Let $s_1,s_2\in\{1,\ldots,\min(n,p)\}$ with $s_1+s_2\le n$, $\lambda\in\Lambda(s_1,s_2)$ and $\mu\in X^+$. Then
$$(T(\lambda):\nabla(\mu))=(T(\lambda):\Delta(\mu))=
\begin{cases}
1\text{\ if $\mu\preceq\lambda$ and $c_{\lambda\mu}$ is oriented,}\\
0\text{\ otherwise.}
\end{cases}
$$
\end{thm}
\begin{proof}
By Proposition~\ref{prop.linkage}(ii) we may assume $\mu\preceq\lambda$. The proof is similar to the proof of \cite[Thm~6.1]{LT}, but it is easier, since we only work with caps. The proof is by induction on the number of caps in $c_\lambda$. If there are no caps in $c_\lambda$, then $c_{\lambda\mu}$ is oriented if and only if $\lambda=\mu$, so the result follows from Lemma~\ref{lem.JSF-arrows}(ii). Otherwise, we choose a cap which has no cap inside it. We will transform this cap to a cap with consecutive end nodes via a sequence of moves which preserve the orientedness of $c_{\lambda\mu}$ and the multiplicity $(T(\lambda):\Delta(\mu))$. We will always move the right end node of the cap one step towards the other end node. In the proof below we will make use of two basic facts. Let $t_1,t_2\in\{1,\ldots,n\}$ with $t_1+t_2\le n$. Firstly, if $\nu\in X^+$ and $\nu'\in{\rm Supp}_h(\nu)$, $h\in\{1,2\}$, with $l(\nu^i),l(\nu'{}^i)\le t_i$ for all $i\in\{1,2\}$, then the $(t_1,t_2)$-arrow diagram of $\nu'$ is obtained from that of $\nu$ by moving one arrow in the $(t_1,t_2)$-arrow diagram of $\nu$ one step: to the right if $h=1$ and to the left if $h=2$. Secondly, if $\nu\in X^+$ and $\nu'\in X^+\cap W_p\cdot\nu$ with $l(\nu^i),l(\nu'{}^i)\le t_i$ for all $i\in\{1,2\}$, then the $(t_1,t_2)$-arrow diagrams of $\nu$ and $\nu'$ have the same number of arrows at each node.

First we prove a general property of the moves we will make. Let $\lambda\in\Lambda(s_1,s_2)$ and $\lambda'\in{\rm Supp}_h(\lambda)\cap\Lambda(s_1,s_2)$, $h\in\{1,2\}$, such that the move $\lambda\mapsto\lambda'$ does not cross or pass a wall. Now let $\nu\in\Lambda(s_1,s_2)\cap W_p\cdot\lambda$ and $\nu'\in{\rm Supp}_h(\nu)\cap W_p\cdot\lambda'$. We show that $\nu'\in\Lambda(s_1,s_2)$. The move from the arrow diagram of $\nu$ to that of $\nu'$ goes between the same nodes as the move $\lambda\mapsto\lambda'$. Assume $l(\nu'{}^1)=s_1+1$. Then $l(\nu^1)=s_1<n-s_2$ and there is no $\land$ immediately to the right of the wall below the line. We temporarily move this wall one step to the left creating a new $\land$ immediately to the right of the new wall.\footnote{At the node of the new $\land$ there may be one other $\land$ and there may be a cap of $c_\lambda$ passing or crossing the new wall.} The move from the arrow diagram of $\nu$ to that of $\nu'$ would move this new arrow one step to the right and therefore cross the original wall. But then the move $\lambda\mapsto\lambda'$ would also cross or pass the original wall. 
This is impossible, so $l(\nu')\le s$. If $\nu'{}^1_1=p-s_1+1$, then $\nu_1=p-s_1$ and the move $\nu\mapsto\nu'$ would pass or cross the wall. This would then also hold for the move $\lambda\mapsto\lambda'$ which is impossible. So $l(\nu'{}^1)\le s_1\le p-\nu'{}^1_1$. The proof that $l(\nu'{}^2)\le s_2\le p-\nu'{}^2_1$ is completely analogous. We conclude that $\nu'\in\Lambda(s_1,s_2)$.

From now on we assume that the nodes are cyclically shifted such that at least one of the walls determined by $s_1$ and $s_2$ is between the first and last node.
When, for a label $a$, we write $a-1$ this is understood to be $p-1$ when $a=0$.
If $\lambda=
\xy
(0,0)="a1";(0,0)*{\cdots};
(5,0)="a2";(5,1)*{\vee};
(10,0)="a3";(10,0)*{\cdots};
(15,0)="a4";(15,0)*{\bullet};
(20,0)="a5";(20,-1)*{\land};(20,-4)*{a};
(25,0)="a6";(25,0)*{\cdots};(25,-6.5)*{\ }; 
"a5";"a2"**\crv{(20,6)&(5,8)}; 
\endxy$,
we choose
$\lambda'=
\xy
(0,0)="a1";(0,0)*{\cdots};
(5,0)="a2";(5,1)*{\vee};
(10,0)="a3";(10,0)*{\cdots};
(15,0)="a4";(15,-1)*{\land};
(20,0)="a5";(20,0)*{\bullet};(20,-3)*{a};
(25,0)="a6";(25,0)*{\cdots};(25,-6.5)*{\ }; 
"a4";"a2"**\crv{(15,6)&(5,7)}; 
\endxy\in{\rm Supp}_2(\lambda)$,
and we put $\Lambda=\Lambda(s_1,s_2)\cap W_p\cdot\lambda$ and $\Lambda'=\Lambda(s_1,s_2)\cap W_p\cdot\lambda'$.
Let $\nu\in\Lambda$. Assume $\nu'\in{\rm Supp}_2(\nu)\cap W_p\cdot\lambda'$. Then we have seen that $\nu'\in\Lambda(s_1,s_2)$.
Moreover, the move $\nu\mapsto\nu'$ moves the arrow at the $a$-node to the $(a-1)$-node.
So the property $\nu'\in{\rm Supp}_2(\nu)\cap W_p\cdot\lambda'$ determines a map $\nu\mapsto\nu':\Lambda\to\Lambda(s_1,s_2)$ given by
\parbox[c][.9cm]{3.2cm}{$\begin{smallmatrix}
\cdots&{\rm o}&\land&\cdots&\mapsto&\cdots&\land&{\rm o}&\cdots\\
\cdots&{\rm o}&\vee&\cdots&\mapsto&\cdots&\vee&{\rm o}&\cdots\\
&&a&&&&&a&
\end{smallmatrix}
$}.
This map clearly preserves the order $\preceq$ and $W_p$-conjugacy (under the dot action), so it has its image in $\Lambda'$.
Similarly, the property $\nu\in {\rm Supp}_1(\nu')\cap W_p\cdot\lambda$ determines a map $\nu'\mapsto\nu:\Lambda'\to\Lambda(s_1,s_2)$ given by reading the above rule in the opposite direction and this map preserves $\preceq$ and $W_p$-conjugacy. So these maps are each others inverse and Proposition~\ref{prop.trans_equivalence} gives that $(T(\lambda):\Delta(\mu))=(T(\lambda'):\Delta(\mu'))$. Furthermore, since $\times$'s and empty nodes don't really play a role in the cap diagram, it is obvious that $c_{\lambda'\mu'}$ is oriented if and only if $c_{\lambda\mu}$ is oriented.
When $\lambda=
\xy
(0,0)="a1";(0,0)*{\cdots};
(5,0)="a2";(5,1)*{\vee};
(10,0)="a3";(10,0)*{\cdots};
(15,0)="a4";(15,1)*{\vee};(15,-1)*{\land};
(20,0)="a5";(20,-1)*{\land};(20,-4)*{a};
(25,0)="a6";(25,0)*{\cdots}; 
"a5";"a2"**\crv{(20,6)&(5,8)}; 
\endxy\ ,$
we choose
$\lambda'=
\xy
(0,0)="a1";(0,0)*{\cdots};
(5,0)="a2";(5,1)*{\vee};
(10,0)="a3";(10,0)*{\cdots};
(15,0)="a4";(15,-1)*{\land};
(20,0)="a5";(20,1)*{\vee};(20,-1)*{\land};(20,-4)*{a};
(25,0)="a6";(25,0)*{\cdots}; 
"a4";"a2"**\crv{(15,6)&(5,7)}; 
\endxy\in{\rm Supp}_1(\lambda)$.
We define $\Lambda$ and $\Lambda'$ as before and similar arguments as above give a bijection $\Lambda\to\Lambda'$ given by
\parbox[c][.9cm]{3.2cm}{$\begin{smallmatrix}
\cdots&\times&\land&\cdots&\mapsto&\cdots&\land&\times&\cdots\\
\cdots&\times&\vee&\cdots&\mapsto&\cdots&\vee&\times&\cdots\\
&&a&&&&&a&
\end{smallmatrix}
$}
with the same properties as before.
In this case we move a unique arrow from the $(a-1)$-node to the $a$-node to go from $\nu$ to $\nu'$, although we think of the move as the arrow at the $(a-1)$-node moving past the $\times$.
So in this case Proposition~\ref{prop.trans_equivalence} again gives that $(T(\lambda):\Delta(\mu))=(T(\lambda'):\Delta(\mu'))$. Furthermore, we again have that $c_{\lambda'\mu'}$ is oriented if and only if $c_{\lambda\mu}$ is oriented.

Now we are reduced to the case that the cap has consecutive end nodes. So $\lambda=
\xy
(0,0)="a1";(0,0)*{\cdots};
(5,0)="a2";(5,1)*{\vee};
(10,0)="a3";(10,-1)*{\land};(10,-4)*{a};
(15,0)="a4";(15,0)*{\cdots};
"a3";"a2"**\crv{(10,5)&(5,5)}; 
\endxy\ $. Now we choose
$\lambda'=
\xy
(0,0)="a1";(0,0)*{\cdots};
(5,0)="a3";(5,1)*{\vee};(5,-1)*{\land};
(10,0)="a3";(10,0)*{\bullet};(10,-4)*{a};
(15,0)="a4";(15,0)*{\cdots};(15,-7)*{\ }; 
\endxy\in {\rm Supp}_2(\lambda)$.
Define $\Lambda$ and $\Lambda'$ as before. Let $\nu\in\Lambda$ and $\nu'\in {\rm Supp}_2(\nu)\cap W_p\cdot\lambda'$.
Then $\nu'\in\Lambda(s_1,s_2)$ as we have seen, and $\nu'$ is obtained from $\nu$ by moving the arrow at the $a$-node to the $(a-1)$-node.
Furthermore, this move can only be done when the arrows at the $(a-1)$-node and $a$-node are not both $\vee$ or both $\land$, i.e. when a cap connecting the two nodes is oriented.
Let us denote the set of $\nu\in\Lambda$ with this property by $\tilde\Lambda$. Then we obtain a map $\nu\mapsto\nu':\tilde\Lambda\to\Lambda(s_1,s_2)$ given by
\parbox[c][.9cm]{3.55cm}{$\begin{smallmatrix}
\cdots&\land&\vee&\cdots\\
\cdots&\vee&\land&\cdots
\end{smallmatrix}
\mapsto
\begin{smallmatrix}
\cdots&\times&{\rm o}&\cdots&\\
\end{smallmatrix}
$}
and it not hard to see that this map preserves $\preceq$ and $W_p$-conjugacy and therefore has its image in $\Lambda'$.\footnote{For the preservation of $\preceq$ one can use functions like the $l_i(\lambda,\mu)$ in \cite[Sect~8]{CdV} and \cite[Sect~5]{BS}.} 

Now let $\nu'\in\Lambda'$ and $\nu\in{\rm Supp}_1(\nu')\cap W_p\cdot\lambda$. Then $\nu\in\Lambda(s_1,s_2)$ by the general fact at the start of the proof, and we see that $\nu=\nu^\pm\in\tilde\Lambda$, where $\nu^+$ resp. $\nu^-$ is obtained from $\nu'$ by moving the $\land$ resp. $\vee$ at the $(a-1)$-node to the $a$-node. So the above map has image equal to $\Lambda'$. Furthermore, it is easy to see that $\eta\preceq\nu$ implies $\eta^-\preceq\nu^-$ and $\eta^+\preceq\nu^+$. By Lemma~\ref{lem.JSF-arrows}(iii) we have that $\Hom_G(\nabla(\nu^+),\nabla(\nu^-))\ne0$.
Since $\lambda=\lambda^+$ we have by Proposition~\ref{prop.trans_projection} that $$(T(\lambda):\Delta(\mu))=(T_{\lambda'}^\lambda T(\lambda'):\Delta(\mu))=(T(\lambda'):T_\lambda^{\lambda'}\Delta(\mu))=(T(\lambda'):\Delta(\mu'))\,,$$ when $\mu=\mu^\pm$ for some $\mu'\in\Lambda'$, i.e. $\mu\in\tilde\Lambda$, and $0$ otherwise.
Here we used that for any finite dimensional $G$-module $M$ with a Weyl filtration $(M:\Delta(\mu))=\dim\Hom_G(M,\nabla(\mu))$.
Finally, $c_{\lambda\mu}$ is oriented if and only if our cap is oriented in $c_{\lambda\mu}$ and $c_{\lambda'\mu'}$ is oriented. So we can now finish by applying the induction hypothesis,
since $c_{\lambda'}$ has one cap less than the original $c_\lambda$.
\end{proof}

\section{Decomposition numbers}\label{s.dec_num}
Let $\mu\in\Lambda_p$. Choose $s_1,s_2\in\{1,\ldots,\min(n,p)\}$ with $s_1+s_2\le n$ and $\mu\in\Lambda(s_1,s_2)$.
First we define the \emph{cap codiagram} $co_\mu$ of the arrow diagram associated to $\mu\in X^+$ as follows.
We assume that the arrow diagram of $\mu$ is cyclically shifted such that at least one of the walls is between the first and last node.
All caps are clockwise, starting from the leftmost node. We start on the left side of the wall. We form the caps recursively. Find an arrow pair $\land\vee$ that are neighbours in the sense that the only arrows in between are already connected with a cap or are part of an $\times$, and connect them with a cap. Repeat this until there are no more such arrow pairs. Now the unconnected arrows that are not part of an $\times$ form a sequence $\vee\cdots\vee\land\cdots\land$. Note that none of these arrows occur inside a cap. The caps on the right side of the wall are formed in the same way.
For example, when $p=17$, $n=20$, $s_1=8$, $s_2=7$ and $\mu=[8^2643^221,107^242^21]$, then
$co_\mu$ is
$$\xy
(0,0)="a1";(0,1)*{\vee};(0,-4)*{13};
(5,0)="a2";(5,-1)*{\land};
(10,0)="a3";
(15,0)="a4";(15,-1)*{\land};(15,-4)*{16};
(20,0)="a5";(20,1)*{\vee};(20,-4)*{0};
(25,0)="a6";(25,-1)*{\land};
(30,0)="a7";(30,-1)*{\land};
(35,0)="a8";(35,1)*{\vee};
(40,0)="a9";(40,1)*{\vee};(40,-1)*{\land};
(45,0)="a10";
(50,0)="a11";(50,1)*{\vee};
(55,0)="a12";(55,-1)*{\land};(57.5,3)*{\rule[0mm]{.3mm}{6mm}};
(60,0)="a13";(60,1)*{\vee};
(65,0)="a14";
(70,0)="a15";(70,-1)*{\land};
(75,0)="a16";(75,-1)*{\land};
(80,0)="a17";(80,1)*{\vee};(80,-4)*{12};(82.5,-3)*{\rule[0mm]{.3mm}{6mm}};
{"a1";"a17"**@{-}}; 
"a5";"a4"**\crv{(20,6)&(15,6)};
"a8";"a7"**\crv{(35,6)&(30,6)};
"a11";"a6"**\crv{(50,11)&(25,9)};
"a17";"a16"**\crv{(80,6)&(75,6)};
\endxy\ .$$

Let $\lambda\in\Lambda_p$ with $\mu\preceq\lambda$. If necessary, we change $s_1,s_2$ (and the arrow diagram of $\mu$, and $co_\mu$) to make sure that $\lambda\in\Lambda(s_1,s_2)$. Then the arrow diagram of $\lambda$ has its single arrows and its $\times$'s at the same nodes as the arrow diagram of $\mu$.
We assume that the nodes are cyclically shifted such that at least one of the walls determined by $s_1$ and $s_2$ is between the first and last node.
Then we know, by Remark~\ref{rems.preceq}.1, that the arrow diagram of $\lambda$ can be obtained from that of $\mu$ by repeatedly replacing an arrow pair $\land\vee$ to the left or to the right of the wall by the opposite arrow pair.
Now we define the \emph{cap codiagram} $co_{\mu\lambda}$ associated to $\mu$ and $\lambda$ by replacing each arrow in $co_\mu$ by the arrow from the arrow diagram of $\lambda$ at the same node. Put differently, we put the caps from $co_\mu$ on top of the arrow diagram of $\lambda$. We say that $co_{\mu\lambda}$ is \emph{oriented} if all caps in $co_{\mu\lambda}$ are oriented (clockwise or anti-clockwise). It is not hard to show that when $co_{\mu\lambda}$ is oriented, the arrow diagrams of $\mu$ and $\lambda$ are the same at the nodes which are not endpoints of a cap in $co_\mu$.

For example, when $p=5$, $n=7$, $s_1=2$, $s_2=3$ and $\mu=[2,1]$. Then $\rho_{s_1}=s_1'=6$, and $co_\mu$ (shifted) is
$$
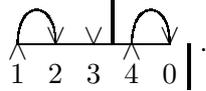
\xy
(0,0)="a1";(0,-1)*{\land};(0,-4)*{1};
(5,0)="a2";(5,1)*{\vee};(5,-4)*{2};
(10,0)="a3";(10,1)*{\vee};(10,-4)*{3};(12.5,3)*{\rule[0mm]{.3mm}{6mm}};
(15,0)="a4";(15,-1)*{\land};(15,-4)*{4};
(20,0)="a5";(20,1)*{\vee};(20,-4)*{0};(22.5,-3)*{\rule[0mm]{.3mm}{6mm}};
{"a1";"a5"**@{-}}; 
"a2";"a1"**\crv{(5,6)&(0,6)}; 
"a5";"a4"**\crv{(20,6)&(15,6)} 
\endxy\ .$$
Consider two dominant weights $\lambda$ with $\mu\preceq\lambda$: $[31,21]$ and $[32,21^2]$
with (shifted) arrow diagrams $\vee\land\vee\vee\land$ and $\vee\vee\land\vee\land$.
Only for the first $co_{\mu\lambda}$ is oriented.

\begin{thm}\label{thm.dec_num}
Let $s_1,s_2\in\{1,\ldots,\min(n,p)\}$ with $s_1+s_2\le n$, $\lambda\in\Lambda(s_1,s_2)$ and $\mu\in X^+$. Then
$$[\nabla(\lambda):L(\mu)]=[\Delta(\lambda):L(\mu)]=
\begin{cases}
1\text{\ if $\mu\preceq\lambda$ and $co_{\mu\lambda}$ is oriented,}\\
0\text{\ otherwise.}
\end{cases}
$$
\end{thm}
\begin{proof}
The proof is by induction on the number of caps in $co_\mu$ and is completely analogous to the proof of Theorem~\ref{thm.filt_mult}.
The role of $\lambda$ is now played by $\mu$. We leave the details to the reader. The final argument involving the projection is as in the proof of \cite[Thm~7.1]{LT}.
\end{proof}

For $s\in\{1,\ldots,\min(n,p)\}$ with $2s\le n$ define the involution $\dagger$ on $\Lambda(s,s)$ by letting $\lambda^\dagger$ be the dominant weight whose arrow diagram is obtained from that of $\lambda$ by replacing all single arrows by their opposite. Note that $\dagger$ reverses the order $\preceq$.
\begin{cornn}
Let $s\in\{1,\ldots,\min(n,p)\}$ with $2s\le n$ and let $\lambda,\mu\in\Lambda(s,s)$. Then $[\Delta(\lambda):L(\mu)]=(T(\mu^\dagger):\nabla(\lambda^\dagger))$.
\end{cornn}
\begin{proof}
This follows from Theorems~\ref{thm.filt_mult} and \ref{thm.dec_num}, since $co_{\mu\lambda}$ is obtained form $c_{\mu^\dagger\lambda^\dagger}$ by replacing all single arrows by their opposite.
\end{proof}

\begin{rem}
In view of \cite[Lem A4.6]{D2} and the above corollary it is natural to conjecture that, for $\Lambda$ the intersection of $\Lambda(s,s)$ with a $W_p$-orbit under the dot action, the algebra $(O_{\Lambda^\dagger}(k[G])^*,\preceq)$ is the Ringel dual of $(O_\Lambda(k[G])^*,\preceq)$.
\end{rem}

\section{The walled Brauer algebra and the rational Schur functor}\label{s.rat_Schur_functor}
We want to relate our results for the general linear group to the walled Brauer algebra $B_{r,s}(n)$. This is natural since $\GL_n$ and $B_{r,s}(n)$ are each others centraliser on mixed tensor space $V^{\ot r}\ot(V^*)^{\ot s}$, see \cite[Sect~4]{T} for the characteristic $p$ case.
For this we will need to introduce the rational Schur functor $f_{rat}$ from a certain category of $G$-modules to the category of finite dimensional modules for the walled Brauer algebra.
In Section~\ref{ss.rat_schur_walled_brauer_alg} we briefly discuss the rational Schur algebra and the walled Brauer algebra. In Section~\ref{ss.walled_brauer_mods} we introduce Specht, permutation and Young modules for the walled Brauer algebra and certain twisted analogues. In section~\ref{ss.rat_schur_func} we introduce the rational Schur functor and derive its main properties. The main results are Theorem~\ref{thm.rat_schur} and Proposition~\ref{prop.rat_mult}. Combining Proposition~\ref{prop.rat_mult} with Theorem~\ref{thm.filt_mult} we obtain as a corollary the decomposition numbers of the walled Brauer algebra when $p$ is bigger than the greatest hook length in the partitions involved. In Section~\ref{ss.further_results} we prove some results for the inverse rational Schur functor and for Young modules. In the case of the symplectic group and the Brauer algebra all this was done in \cite[Sect~1,2]{DT}. We follow the treatment there closely.

\subsection{The rational Schur algebra and the walled Brauer algebra}\label{ss.rat_schur_walled_brauer_alg}
Let $r,s$ be integers $\ge0$. For any $\delta\in k$ one has the \emph{walled Brauer algebra} $B_{r,s}(\delta)$; see e.g. \cite{CdVDM} or \cite{T} for a definition. Recall that it is defined as a subalgebra of the Brauer algebra $B_{r+s}(\delta)$. In each Brauer diagram one draws a wall that goes between the first $r$ nodes and the last $s$ nodes in the top row and between the first $r$ nodes and the last $s$ nodes in the bottom row. Then $B_{r,s}(\delta)$ is spanned by the {\it walled Brauer diagrams} which are the Brauer diagrams in which each horizontal edge, i.e. an edge joining two vertices in the same row, crosses the wall and each vertical edge, i.e. an edge joining a vertex in the top row to one in the bottom row, is on one side of the wall. This also makes sense for $\delta$ an integer, since we can replace that integer by its natural image in $k$. The walled Brauer algebra is a cellular algebra, see e.g. \cite[Thm~2.7]{CdVDM}. Put $V^{r,s}=V^{\ot r}\ot(V^*)^{\ot s}$. Then we have natural homomorphisms $k\Sym_r\to\End_G(V^{\ot r})$ and $B_{r,s}(n)\to\End_G(V^{r,s})$. The action of the symmetric group $\Sym_r$ is by permutation of the factors, the action of $B_{r,s}(n)$ is explained in \cite[p 564,565]{BCHLLS} and \cite[p1220]{T}. Using classical invariant theory one can then show that these homomorphisms are surjective and that they are injective in case $n\ge r$ and $n\ge r+s$, respectively; see \cite{DeCProc} and \cite[Thm~4.1]{T}. Let $S(n,r)$ and $S(n;r,s)$ be the spans of the representing automorphisms of $G$ in $\End(V^{\ot r})$ and $\End(V^{r,s})$ respectively. Then these are algebras and the natural embeddings $S(n,r)\to\End_{k\Sym_r}(V^{\ot r})$ and $S(n;r,s)\to\End_{B_{r,s}(n)}(V^{r,s})$ are isomorphisms; see \cite[(2.6c)]{Green} and \cite[Thm~4.1]{T}. The algebra $S(n,r)$ is the {\it Schur algebra}, see \cite{Green}, and $S(n;r,s)$ is the {\it rational Schur algebra} introduced in \cite{DD}, see also \cite{D3}. Both algebras are generalised Schur algebras, see \cite[Ch~A]{Jan}. For $S(n,r)$ the corresponding set of dominant weights is the set of partitions of $r$ of length $\le n$ and for $S(n;r,s)$ it is
$$\Lambda_{r,s}:=\Big\{\lambda=[\lambda^1,\lambda^2]\in X^+\,\Big|\,|\lambda^1|=r-t\text{\ and\ }|\lambda^2|=s-t\text{\ for some $t\in\mb N$}\Big\}\,.$$

The following lemma is well-known; it will be used in Section~\ref{ss.rat_schur_func}.
\begin{lem}\label{lem.directsummand}\
\begin{enumerate}[{\rm(i)}]
\item Let $M$ be a finite dimensional vector space over $k$. The $k\GL(M)$-module $M$ is a direct summand of $M\ot M^*\ot M$ and if $\dim M\ne0$ in $k$, then the trivial $k\GL(M)$-module $k$ is a direct summand of $M\ot M^*$.
\item Let $H$ be a group and let $M$ be a finite dimensional $kH$-module. Let $r,s,t$ be integers with $r,s\ge t\ge 1$. Then $M^{\ot r-t}\ot (M^*)^{\ot s-t}$ is a direct summand of $M^{\ot r}\ot (M^*)^{\ot s}$ if $r-t$ and $s-t$ are not both $0$ or if $\dim M\ne0$ in $k$.
\end{enumerate}
\end{lem}
\begin{proof}
(i) is \cite[Lem 1.1(i)]{DT} and (ii) follows from (i) by induction.
\end{proof}

\subsection{Modules for the walled Brauer algebra}\label{ss.walled_brauer_mods}
\begin{notnn}
Put $\Sym_{r,s}=\Sym_r\times\Sym_s$. In what follows, $r,s,t$ are integers with $r,s\ge t\ge0$ and we put $r'=r-t$ and $s'=s-t$.
\end{notnn}
Let $\delta\in k$. For any integer $i\ge0$, let $I_{t,i}$ be the left ideal of the walled Brauer algebra $B_{r,s}=B_{r,s}(\delta)$ spanned by the diagrams of which the bottom row has at least $t+i$ horizontal edges, $t$ of which join, for $1\le j\le t$, the $j$-th node from the right before the wall to the $j$-th node from the right after the wall. Put $I_t:=I_{s,0}$, $Z_{t,i}:=I_{t,i}/I_{t,i+1}$ and $Z_t=Z_{t,0}$. Note that $I_{t,i}=Z_{t,i}=0$ if $t+i>\min(r,s)$. The group $\Sym_{r',s'}$ acts on $I_t$ from the right by permuting the first $r'$ nodes before the wall and the first $s'$ nodes after the wall of the bottom row of a diagram. Thus $I_t$ and $Z_t$ are $(B_{r,s}(\delta), k\Sym_{r',s'})$-bimodules. Furthermore $Z_t$ is a free right $k\Sym_{r',s'}$-module which has as a basis the canonical images of the diagrams in which the vertical edges do not cross and of which the bottom row has precisely $t$ horizontal edges which join, for $1\le j\le t$, the $j$-th node from the right before the wall to the $j$-th node from the right after the wall. One easily checks that there are $$\binom{r}{r'}\binom{s}{s'}t!=\binom{r}{t}\binom{s}{t}t!$$ such diagrams.

For $\mu$ a partition of $r$ let $S(\mu)$, $M(\mu)$ and $Y(\mu)$ be the Specht module, permutation module and Young module of $k\Sym_r$ associated to $\mu$. If ${\rm char}\,k=0$, then $S(\mu)$ is irreducible and we also denote it by $D(\mu)$.
If ${\rm char}\,k=p>0$ and $\mu$ is $p$-regular, then $S(\mu)$ has a simple head and we denote it by $D(\mu)$. Denote the sign representation of $k\Sym_r$ by $k_{\rm sg}$.

If $A$ and $B$ are $k$-algebras, $M$ is an $A$-module and $N$ is a $B$-module, then $M\boxtimes N$ denotes $M\ot N$ endowed with its natural structure of $A\ot B$-module. Let $\lambda^1$ and $\lambda^2$ be partitions of $r$ and $s$ respectively. When it is clear that we are dealing with a $k\Sym_{r,s}$-module, we denote $k_{\rm sg}\boxtimes k_{\rm sg}$ simply by $k_{\rm sg}$.
Following \cite[Sect~3]{CdVDM}, we define the {\it Specht (or cell) module} $\mc S(\lambda^1,\lambda^2)$ and {\it twisted Specht (or cell) module} $\widetilde{\mc S}(\lambda^1,\lambda^2)$ for the walled Brauer algebra by
\begin{align*}
\mc S(\lambda^1,\lambda^2):=&\,Z_t\ot_{k\Sym_{r',s'}}(S(\lambda^1)\boxtimes S(\lambda^2))\text{\quad and}\\
\widetilde{\mc S}(\lambda^1,\lambda^2):=&\,Z_t\ot_{k\Sym_{r',s'}}\big(k_{\rm sg}\ot (S(\lambda^1)\boxtimes S(\lambda^2))\big).
\end{align*}
By the above, $\dim\mc S(\lambda^1,\lambda^2)=\dim\widetilde{\mc S}(\lambda^1,\lambda^2)=\binom{r}{t}\binom{s}{t}t!\dim S(\lambda^1)\dim S(\lambda^2)$.
By \cite[Rem.~6.4]{Green} we have $k_{\rm sg}\ot S(\mu)\cong S(\mu^T)^*$, where $\mu^T$ denotes the transpose of $\mu$. If ${\rm char}\,k=0$ or $\lambda^1,\lambda^2$ are $p$-cores, then $S(\lambda^h)^*\cong S(\lambda^h)$ for all $h\in\{1,2\}$ and $\widetilde{\mc S}(\lambda^1,\lambda^2)\cong{\mc S}(\lambda^1{}^T,\lambda^2{}^T)$.
If $\lambda^1,\lambda^2$ are $p$-regular and $\ne\emptyset$ in case $r=s\ge1$ and $\delta=0$, then ${\mc S}(\lambda^1,\lambda^2)$ and $\widetilde{\mc S}(\lambda^1,\lambda^2)$ have a simple head which we denote by ${\mc D}(\lambda^1,\lambda^2)$ and $\widetilde{\mc D}(\lambda^1,\lambda^2)$, see \cite[Thm~2.7]{CdVDM}. Whenever we write ${\mc D}(\lambda^1,\lambda^2)$ or $\widetilde{\mc D}(\lambda^1,\lambda^2)$ for some $p$-regular $\lambda^1,\lambda^2$, we assume that $\lambda^1,\lambda^2\ne\emptyset$ in case $r=s\ge1$ and $\delta=0$.

As Hartmann and Paget \cite[Sect~6]{HarPag} did for the Brauer algebra, we define the {\it permutation module} $\mc M(\lambda^1,\lambda^2)$ and the {\it twisted permutation module} $\widetilde{\mc M}(\lambda^1,\lambda^2)$ for the walled Brauer algebra by
\begin{align*}
\mc M(\lambda^1,\lambda^2):=&\,{\rm Ind}^{B_{r,s}}_{k\Sym_{r',s'}}(M(\lambda^1)\boxtimes M(\lambda^2))\text{\quad and}\\
\widetilde{\mc M}(\lambda^1,\lambda^2):=&\,{\rm Ind}^{B_{r,s}}_{k\Sym_{r',s'}}\big(k_{\rm sg}\ot (M(\lambda^1)\boxtimes M(\lambda^2))\big).
\end{align*}
Here ${\rm Ind}^{B_{r,s}}_{k\Sym_{r',s'}}$ is defined by ${\rm Ind}^{B_{r,s}}_{k\Sym_{r',s'}}M=I_t\ot_{k\Sym_{r',s'}}M$ for any $k\Sym_{r',s'}$-module $M$.
Note that ${\mc M}(1^r,1^s)\cong B_{r,s}$ and $\widetilde{\mc M}(1^r,1^s)\cong B_{r,s}$, since $M(1^r,1^s)=k\Sym_{r,s}$, $k_{\rm sg}\ot k\Sym_{r,s}\cong k\Sym_{r,s}$ as $k\Sym_{r,s}$-modules and $I_0=B_{r,s}$. 

Finally, we define the {\it Young module} $\mc Y(\lambda^1,\lambda^2)$ and the {\it twisted Young} module $\widetilde{\mc Y}(\lambda^1,\lambda^2)$ for the walled Brauer algebra as the unique indecomposable summand of $\mc M(\lambda^1,\lambda^2)$ resp. $\widetilde{\mc M}(\lambda^1,\lambda^2)$ which surjects onto $Z_t\ot_{k\Sym_{r',s'}}(Y(\lambda^1)\boxtimes Y(\lambda^2))$ resp. $Z_t\ot_{k\Sym_{r',s'}}\big(k_{\rm sg}\ot(Y(\lambda^1)\boxtimes Y(\lambda^2))\big)$; compare \cite[Def.~15]{HarPag}. 

Let $i$ be an integer $\ge0$. We denote the diagonal copy of $\Sym_i$ in $\Sym_{i,i}$ by $D_i$. We consider $\Sym_{i,i}$ and $D_i$ as embedded in $\Sym_{r',s'}$ via the embedding $\Sym_{r'-i,s'-i}\times\Sym_{i,i}\subseteq\Sym_{r',s'}$. From the proof of \cite[Prop.~23]{HarPag} in the Brauer algebra case we have

\begin{prop}[{cf. \cite[Prop~1.1]{DT}}]\label{prop.filtration}
Let $M$ be a $k\Sym_{r',s'}$-module.
\begin{enumerate}[{\rm(i)}]
\item $P:={\rm Ind}^{B_{r,s}}_{k\Sym_{r',s'}}M$ has a descending filtration $P=P_0\supseteq P_1\supseteq\cdots$ such that $P_i=0$ for $i>t$ and $P_i/P_{i+1}\cong Z_{t,i}\ot_{k\Sym_{r',s'}}M$ for $i\ge0$.
\item $Z_{t,i}\ot_{k\Sym_{r',s'}}M\cong Z_{t+i}\ot_{k\Sym_{r'-i,s'-i}}M_{D_i}$ for $i\le t$, where $M_{D_i}$ is the largest trivial $D_i$-module quotient of $M$.
\end{enumerate}
\end{prop}
The filtration of ${\rm Ind}^{B_{r,s}}_{k\Sym_{r',s'}}M=I_t\ot_{k\Sym_{r',s'}}M$ is constructed as follows. Let $I_t(i)$ be the subspace of $I_t$ spanned by the diagrams of which the bottom row has exactly $t+i$ horizontal edges, $t$ of which join, for $1\le j\le t$, the $j$-th node from the right before the wall to the $j$-th node from the right after the wall. Then $I_{t,i}=\bigoplus_{j\ge i}I_t(j)$. Since each $I_t(i)$ is stable under the right action of $\Sym_{r',s'}$ on $I_t$, we have ${\rm Ind}^{B_{r,s}}_{k\Sym_{r',s'}}M=\bigoplus_{i\ge0}(I_t(i)\ot_{k\Sym_{r',s'}}M)$. Now we put $P_i=\bigoplus_{j\ge i}(I_t(j)\ot_{k\Sym_{r',s'}}M)\cong I_{t,i}\ot_{k\Sym_{r',s'}}M$ and observe that $P_i$ is a $B_{r,s}$-submodule of $P$.

The following result shows that we can restrict to the case that $\delta$ lies in the prime field. It is the analogue of \cite[Prop~1.2]{DT}.
It can be proved in the same way where the role of \cite[Prop.~6.1]{CdVM2} is now played by \cite[Cor~4.3]{CdVDM}.
\begin{prop}[{cf.~\cite[Cor~4.3]{CdVDM}}]\label{prop.notinFp}
Assume that $\delta$ does not lie in the prime field. Put $n_i=\binom{r}{i}\binom{s}{i}i!$. Then we have an algebra isomorphism
$$B_{r,s}(\delta)\cong\bigoplus_{i=0}^{\min(r,s)}\Mat_{n_i}(k\Sym_{{r-i},s-i}).$$
\end{prop}
In the remainder of this subsection we assume that $\delta=n$ and that $n\ge r+s$.
The {\it contravariant dual} $M^\circ$ of a finite dimensional $G$-module $M$ is defined as the dual vector space of $M$ with action $(g\cdot f)(x)=f(g^Tx)$.
As is well-known, $L(\lambda)^\circ\cong L(\lambda)$, $\nabla(\lambda)^\circ\cong\Delta(\lambda)$ and $\Delta(\lambda)^\circ\cong\nabla(\lambda)$ for all $\lambda\in X^+$.
So $V^\circ\cong V$, $(V^*)^\circ\cong V^*$ and therefore $(V^{r,s})^\circ\cong V^{r,s}$. Put differently, the standard inner products on $V$ and $V^*$ induce a nondegenerate bilinear form $(-,-)$ on $V^{r,s}$ which is contravariant: $(gu,v)=(u,g^Tv)$ for all $u,v\in V^{r,s}$ and all $g\in G$.
This implies that $S(n;r,s)$ is stable under the transpose map of $\End_k(V^{r,s})$ given by this form. We can use this transpose map to define the dual of any $S(n;r,s)$ module $M$ which of course identifies with $M^\circ$. Recall that $B_{r,s}$ has a standard anti-automorphism $\iota$ that flips a diagram over the horizontal axis. One easily checks that $(bu,v)=(u,\iota(b)v)$ for all $u,v\in V^{r,s}$ and all $b\in B_{r,s}$. 
This means that the $B_{r,s}$-module $V^{r,s}$ is self-dual.

Using the description of the invariants of vectors and covectors for $\GL_n$ we see that $\Hom_G(V^{r_2,s_2},V^{r_1,s_1})$ has a basis indexed by $((r_1,s_1),(r_2,s_2))$-diagrams. These are diagrams which are graphs whose vertices are arranged in two rows, $r_1+s_1$ in the top row and $r_2+s_2$ in the bottom row with a wall which goes between the first $r_1$ nodes and the last $s_1$ nodes at the top, and between the first $r_2$ nodes and the last $s_2$ nodes at the at the bottom. 
The edges form a matching of the vertices in pairs such that the horizontal edges cross the wall and the vertical edges don't.
See e.g. the proof of \cite[Thm~4.1]{T} and the preceding paragraph.
The horizontal edges in the bottom row correspond to contractions by means of the canonical bilinear form and the horizontal edges in the top row correspond to ``multiplications" by the invariant $\sum_{i=1}^ne_i\ot e_i^*$, where the $e_i$ and $e_i^*$ are the elements of the standard basis of $V$ and its dual basis. In the proofs of Lemmas~\ref{lem.surjective} and \ref{lem.tensor} below we will use these diagram bases.

The diagrams that form a basis of $I_t$ are in 1-1 correspondence with the $((r,s),(r',s'))$-diagrams: just omit in the bottom row the last $t$ nodes before the wall and the last $t$ nodes after the wall, and the edges which have these nodes as endpoints. So the canonical isomorphism $B_{r,s}\stackrel{\sim}{\to}\End_G(V^{r,s})$ induces a canonical isomorphism
$$I_t\stackrel{\sim}{\to}\Hom_G(V^{r',s'},V^{r,s})$$ of $(B_{r,s},k\Sym_{r',s'})$-bimodules. 
The vector space $\Hom_G(V^{r,s},V^{r',s'})$ has a natural $(k\Sym_{r',s'},B_{r,s})$-bimodule structure and therefore, by means of the standard anti-automorphisms of $\Sym_{r',s'}$ and $B_{r,s}$, also a natural $(B_{r,s},k\Sym_{r',s'})$-bimodule structure. Composing the above isomorphism with the transpose map $\Hom_G(V^{r',s'},V^{r,s})\to\Hom_G(V^{r,s},V^{r',s'})$, using contravariant duals, we obtain a canonical isomorphism
\begin{equation}\label{eq.I_tiso}
\varphi:I_t\stackrel{\sim}{\to}\Hom_G(V^{r,s},V^{r',s'})
\end{equation}
of $(B_{r,s},k\Sym_{r',s'})$-bimodules, which induces an isomorphism
\begin{equation}\label{eq.Z_tiso}
Z_t\stackrel{\sim}{\to}\Hom_G(V^{r,s},V^{r',s'})/\varphi(I_{t,1})
\end{equation}
of $(B_{r,s},k\Sym_{r',s'})$-bimodules.

\subsection{The rational Schur functor}\label{ss.rat_schur_func}
For a finite dimensional algebra $A$ over $k$, we denote the category of finite dimensional $A$-modules by ${\rm mod}(A)$.
Assume that $n\ge r,s\ge0$. The Schur functor $f:{\rm mod}(S(n,r))\to{\rm mod}(k\Sym_r)$ can be defined by
\begin{align*}
&f(M)=\Hom_{S(n,r)}(V^{\ot r},M)=\Hom_G(V^{\ot r},M).
\end{align*}
Here the action of the symmetric group comes from the action on $V^{\ot r}$ and we use the inversion to turn right modules into left modules. An equivalent definition is: $f(M)=M_{\varpi_r}$, the weight space corresponding to the weight $\varpi_r=1^r=(1,1,\ldots,1)\in\mb Z^r\subseteq\mb Z^n$; see \cite{Green}. An isomorphism
\begin{align}\label{eq.schuriso}
\Hom_G(V^{\ot r},M)\stackrel{\sim}{\to} M_{\varpi_r}
\end{align}
is given by $u\mapsto u(e_1\ot e_2\ot\cdots\ot e_r)$. This can be deduced from \cite[6.2g Rem.~1 and 6.4f]{Green}. We have embeddings $\Sym_r\subseteq\Sym_n\subseteq N_G(T)$, where the second embedding is by permutation matrices. Then $\varpi_r$ is fixed by $\Sym_r$, so there is an action of $\Sym_r$ on $M_{\varpi_r}$ for every $S(n,r)$-module $M$. With this action \eqref{eq.schuriso} is $\Sym_r$-equivariant. The inverse Schur functor $g:{\rm mod}(k\Sym_r)\to{\rm mod}(S(n,r))$ can be defined by $g(M)=V^{\ot r}\ot_{k\Sym_r}M$.

Recall that $\breve{\xi}$ denotes the reversed tuple of $\xi\in\mb Z^n$. We can also define $f(M)=M_{\breve{\varpi}_r}$ and then we have an isomorphism
\begin{align*}
\Hom_G(V^{\ot r},M)\stackrel{\sim}{\to} M_{\breve{\varpi}_r}
\end{align*}
given by $u\mapsto u(e_{n-r+1}\ot e_2\ot\cdots\ot e_n)$. In this case $\Sym_r$ is embedded in $\Sym_n$ as $\Sym(\{n-r+1,\ldots,n\})$.

Combining the above two versions of the Schur functor we can form another Schur functor $f^{(2)}:{\rm mod}(S(n,r)\ot S(n,s))\to{\rm mod}(k\Sym_{r,s})$ by
$f^{(2)}(M)=\Hom_{G\times G}(V^{\ot r}\boxtimes V^{\ot s},M)$ and then we have an isomorphism 
\begin{align}\label{eq.schur2iso}
\Hom_{G\times G}(V^{\ot r}\boxtimes V^{\ot s},M)\stackrel{\sim}{\to} M_{(\varpi_r,\breve{\varpi}_s)}
\end{align}
given by $u\mapsto u((e_1\ot\cdots\ot e_r)\ot(e_{n-s+1}\ot\cdots\ot e_n))$.
This isomorphism is $\Sym_{r,s}$ equivariant if we embed $\Sym_{r,s}$ in $\Sym_{n,n}$ by combining the above two types of embeddings.
It is elementary to verify that for $M$ an $S(n,r)$-module and $N$ an $S(n,s)$-module we have $f^{(2)}(M\boxtimes N)=f(M)\boxtimes f(N)$.

We now retain the notation and assumptions of Section~\ref{ss.walled_brauer_mods}. In particular, $n\ge r+s$ and $B_{r,s}=B_{r,s}(n)$. We define the {\it rational Schur functor}
$$f_{rat}:{\rm mod}(S(n;r,s))\to{\rm mod}(B_{r,s})$$ by
\begin{align*}
&f_{rat}(M)=\Hom_{S(n;r,s)}(V^{r,s},M)=\Hom_G(V^{r,s},M).
\end{align*}
Here the action of the $B_{r,s}$ comes from the action on $V^{r,s}$ and we use the standard anti-automorphism of $B_{r,s}$ to turn right modules into left modules. Since $V=\nabla(\ve_1)=\Delta(\ve_1)$ and $V^*=\nabla(-\ve_n)=\Delta(-\ve_n)$ are tilting modules, the same holds for $V^{r,s}$; see e.g. \cite[Prop~E.7]{Jan}. This implies that $f_{rat}$ maps short exact sequences of modules with a good filtration to exact sequences.

We define the {\it inverse rational Schur functor}
$$g_{rat}:{\rm mod}(B_{r,s})\to{\rm mod}(S(n;r,s))$$ by
\begin{align*}
&g_{rat}(M)=V^{r,s}\ot_{B_{r,s}}M.
\end{align*}

By \cite[Thm~2.11]{Rot} we have for $N\in{\rm mod}(B_{r,s})$ and $M\in{\rm mod}(S(n;r,s))$ 
\begin{equation}\label{eq.adjointiso}
\Hom_G(g_{rat}(N),M)\cong\Hom_{B_{r,s}}(N,f_{rat}(M)).
\end{equation}
There is an alternative for $f_{rat}$ and $g_{rat}$:
$$\tilde f_{rat}(M)=V^{r,s}\ot_{S(n;r,s)}M\text{\quad and\quad }\tilde g_{rat}(N)=\Hom_{B_{r,s}}(V^{r,s},N),$$
where we consider $V^{r,s}$ as right $S(n;r,s)$-module via the transpose map of $S(n;r,s)$.
But, by \cite[Lemma~3.60]{Rot}, we have $\tilde{f}_{rat}(M^\circ)\cong f_{rat}(M)^*$ and $\tilde{g}_{rat}(N^*)\cong g_{rat}(N)^\circ$. 
So the results obtained using $\tilde{f}_{rat}$ and $\tilde{g}_{rat}$ can also be obtained by dualizing the results obtained using $f_{rat}$ and $g_{rat}$.

The following lemma is the analogue of \cite[Lem~2.1]{DT} for our situation.
\begin{lem}\label{lem.dimension}
For $\lambda=[\lambda^1,\lambda^2]\in\Lambda_{r,s}$ we have $\dim\Hom_G(\Delta(\lambda),V^{r,s})=$
$$\dim\Hom_G(V^{r,s},\nabla(\lambda))=\binom{r}{t}\binom{s}{t}t!\dim S(\lambda^1)\dim S(\lambda^2).$$
\end{lem}
\begin{proof}
Since $V^{r,s}$ has a good filtration, the dimension of $\Hom_G(\Delta(\lambda),V^{r,s})$ is equal to the multiplicity of $\nabla(\lambda)$ in a good filtration of $V^{r,s}$.
This multiplicity is equal to the coefficient of $\chi(\lambda)$ in an expression of ${\rm ch}\,V^{r,s}$ as a ${\mb Z}$-linear combination of Weyl characters. Similar remarks apply to the dimension of $\Hom_G(V^{r,s},\nabla(\lambda))$. For a partition $\mu$ denote $\dim S(\mu)$ by $d_\mu$.
For $r,s\ge0$ with $r+s\le n$ put $$\psi_{rs}=\sum_{\lambda^1,\lambda^2}d_{\lambda^1}d_{\lambda^2}\chi([\lambda^1,\lambda^2])\,,$$ where the sum is over all partitions $\lambda^1$ of $r$ and $\lambda^2$ of $s$. Then we have to show that for $r,s\ge0$ with $r+s\le n$ we have
$${\rm ch}\,V^{r,s}=\sum_{t=0}^{\min(r,s)}\binom{r}{t}\binom{s}{t}t!\psi_{r-t,s-t}\,.\eqno{(*)}$$
Since ${\rm ch}\,V^{r,0}=\psi_{r,0}$ and ${\rm ch}\,V^{0,s}=\psi_{0,s}$, by classical Schur-Weyl duality, (*) holds when $s=0$ or $r=0$. 
From the rules for induction and restriction for the pair $\Sym_{r-1}\le\Sym_r$ we obtain that, for $\mu$ a partition of $r-1$, $rd_\mu=\sum_\nu d_\nu$, where the sum is over the partitions $\nu$ of $r$ obtained by adding a box to $\mu$ and, for $\mu$ a partition of $r$, $d_\mu=\sum_\nu d_\nu$ , where the sum is over the partitions $\nu$ of $r-1$ obtained by removing a box from $\mu$.
From this and Brauer's formula \cite[II.5.8]{Jan} we obtain for $r\ge1$, $s\ge0$ with $r+s<n$ that
$${\rm ch}\,(V^*)\psi_{r,s}=\chi(-\ve_n)\psi_{r,s}=\psi_{r,s+1}+r\psi_{r-1,s}\,.$$
From this (*) follows easily by induction on $s$.
\end{proof}

Recall that induced modules for a reductive group can be realized in the algebra of regular functions of the group. We embed $G$ into $G\times G$ via
$$A\mapsto(A, (A^{-1})^T).$$
Let $\lambda=[\lambda^1,\lambda^2]\in X^+$ with $|\lambda^1|=r$, $|\lambda^2|=s$. From the fact that $\nabla(\lambda)$ has a bideterminant basis labelled by standard rational bitableaux, see \cite[Thm.~2.2(iii)]{T}, it is clear that restriction of functions induces an epimorphism $\nabla(\lambda^1)\boxtimes\nabla(\lambda^2)\to\nabla(\lambda)$ of $G$-modules.\footnote{We have $\nabla(\lambda^1)\boxtimes\nabla(\lambda^2)=\nabla(\lambda^1)\ot\nabla(-\breve{\lambda^2})$ as $G$-modules, since twisting with the inverse transpose turns $\nabla(\lambda)$ into $\nabla(-\breve{\lambda})$ and $\Delta(\lambda)$ into $\Delta(-\breve{\lambda})$.}
Now we can form a commutative diagram as below where the vertical maps are induced by the restriction of functions $\nabla(\lambda^1)\boxtimes\nabla(\lambda^2)\to\nabla(\lambda)$ and the horizontal maps are evaluation at $(e_1\ot\cdots\ot e_r)\ot(e_{n-s+1}\ot\cdots\ot e_n)$ and $(e_1\ot\cdots\ot e_r)\ot(e^*_{n-s+1}\ot\cdots\ot e^*_n)$.

\begin{equation}\label{eq.rat_schur_iso}
\vcenter{
\xymatrix @R=30pt @C=15pt @M=6pt{
\Hom_{G\times G}(V^{\ot r}\boxtimes V^{\ot s},\nabla(\lambda^1)\boxtimes\nabla(\lambda^2))\ar[r]\ar[d]&\nabla(\lambda^1)_{\varpi_r}\ot\nabla(\lambda^2)_{\breve{\varpi}_s}\ar[d]\\
\Hom_G(V^{r,s},\nabla(\lambda))\ar[r]&\nabla(\lambda)_{[\varpi_r,\varpi_s]}
}}
\end{equation}
Here $\nabla(\lambda)_\mu$ denotes the $\mu$-weight space of $\nabla(\lambda)$ with respect to $T$.
\begin{lem}\label{lem.diagram}\
\begin{enumerate}[{\rm (i)}]
\item Let $M$ be a homogeneous polynomial $T\times T$-module of bidegree $(r,s)$ and let $\mu^1,\mu^2\in\mb N^n$ with $|\mu^1|=r$ and $|\mu^2|=s$, such that for some $u$ we have $\mu^1_i=0$ for all $i>u$ and $\mu^2_i=0$ for all $i\le u$. Then the $(\mu^1,\mu^2)$-weight space of $M$ with respect to $T\times T$ is the same as the $(\mu^1-\mu^2)$-weight space with respect to $T$, embedded in $T\times T$ via $t\mapsto (t,t^{-1})$.
\item Let $\lambda=[\lambda^1,\lambda^2]\in X^+$ with $|\lambda^1|=r$, $|\lambda^2|=s$ and let $\mu^1,\mu^2\in\mb N^n$ with $|\mu^1|=r$ and $|\mu^2|=s$, such that for some $u$ we have $\mu^1_i=0$ for all $i>u$ and $\mu^2_i=0$ for all $i\le u$. Then the restriction of functions induces an isomorphism $(\nabla(\lambda^1)\boxtimes\nabla(\lambda^2))_{\mu^1-\mu^2}\to\nabla(\lambda)_{\mu^1-\mu^2}$ on the $(\mu^1-\mu^2)$-weight spaces for $T$.
\item All maps in \eqref{eq.rat_schur_iso} are isomorphisms.
\end{enumerate}
\end{lem}
\begin{proof}
(i).\ A weight $(\mu^1,\mu^2)$ of $T\times T$ vanishes on $T$ if and only if $\mu^1=\mu^2$. So if $\mu^h$ and $\nu^h$, $h\in\{1,2\}$, are weights of $T$ such that the $\nu^h$ are polynomial, for some $u$ we have $\mu^1_i=0$ for all $i>u$ and $\mu^2_i=0$ for all $i\le u$, $|\nu^h|=|\mu^h|$ for all $h\in\{1,2\}$, and $(\nu^1,\nu^2)|_T=(\mu^1,\mu^2)|_T$, then $(\nu^1,\nu^2)=(\mu^1,\mu^2)$.\\
(ii).\ Clearly $\nabla(\lambda^1)\boxtimes\nabla(\lambda^2)$ induces a surjection on the weight spaces for $T$. So it suffices to show that $(\nabla(\lambda^1)\boxtimes\nabla(\lambda^2))_{\mu^1-\mu^2}$ and $\nabla(\lambda)_{\mu^1-\mu^2}$ have the same dimension. Note that, by (i), $(\nabla(\lambda^1)\boxtimes\nabla(\lambda^2))_{\mu^1-\mu^2}$ is also the $(\mu^1,\mu^2)$-weight space with respect to $T\times T$. By \cite[4.5a]{Green} $\dim\nabla(\lambda^1)_{\mu^1}\boxtimes\nabla(\lambda^2)_{\mu^2}$ is the number of standard $\lambda^1$-tableaux of content $\mu^1$ times the number of standard $\lambda^2$-tableaux of content $\mu^2$.
By \cite[Thm~3.2(iii)]{T} and the definitions on p1215/1216 in \cite{T} $\dim\nabla(\lambda)_{\mu^1-\mu^2}$ is the number of standard rational tableaux of shape $(\lambda^1,\lambda^2)$ and weight $\mu^1-\mu^2$. 
By the proof of (i) any rational tableau $(T^1,T^2)$ of shape $(\lambda^1,\lambda^2)$ and weight $\nu^1-\nu^2=\mu^1-\mu^2$, $\nu^h$ the weight of $T^h$, $h\in\{1,2\}$, must satisfy $(\nu^1,\nu^2)=(\mu^1,\mu^2)$. Because of our condition on $\mu^1$ and $\mu^2$, $T^1$ and $T^2$ have no numbers in common. So $(T^1,T^2)$ is rational standard if and only if $T^1$ and $T^2$ are standard. So the two dimensions are the same.\\
(iii).\ That the horizontal map in the top row of \eqref{eq.rat_schur_iso} is an isomorphism was pointed out before; see \eqref{eq.schur2iso}. The vertical map on the right is an isomorphism by (ii). It follows that the horizontal map in the bottom row is surjective. But then it must be an isomorphism by Lemma~\ref{lem.dimension}. Now the vertical map on the left must also be an isomorphism, since it is a composite of isomorphisms.
\end{proof}

For $\mu\in\mb N^l$ we put $S^\mu V=S^{\mu_1}V\ot\cdots\ot S^{\mu_l}V$ and $\bigwedge{\Nts}^\mu V=\bigwedge{\Nts}^{\mu_1}V\ot\cdots\ot \bigwedge{\Nts}^{\mu_l}V$.

\begin{lem}\label{lem.surjective}
Recall that $r'=r-t$ and $s'=s-t$. The following holds.
\begin{enumerate}[{\rm(i)}]
\item Let $\lambda=[\lambda^1,\lambda^2]\in X^+$ with $|\lambda^1|=r'$, $|\lambda^2|=s'$. Then the canonical homomorphism
$$\Hom_G(V^{r,s},V^{r',s'})\ot_{k\Sym_{r',s'}}\Hom_G(V^{r',s'},\nabla(\lambda))\to\Hom_G(V^{r,s},\nabla(\lambda)),$$
given by composition, is surjective.
\item Let $M$ be an $S(n,r')\ot S(n,s')$-module. The canonical homomorphism
$$\Hom_G(V^{r,s},V^{r',s'})\ot_{k\Sym_{r',s'}}\Hom_{G\times G}(V^{\ot r'}\boxtimes V^{\ot s'},M)\to\Hom_G(V^{r,s},M),$$
given by composition, is an isomorphism if $M$ is a direct sum of direct summands of $V^{\ot r'}\boxtimes V^{\ot s'}$ and it is surjective if $M$ is injective.
\end{enumerate}
\end{lem}
\begin{proof}
(i). By Lemma~\ref{lem.dimension} it suffices to give a family of $\binom{r}{t}\binom{s}{t}t!\dim S(\lambda^1)\dim S(\lambda^2)$ elements of $\Hom_G(V^{r,s},V^{r',s'})\ot_{k\Sym_{r',s'}}\Hom_G(V^{r',s'},\nabla(\lambda))$ which is mapped to an independent family in $\Hom_G(V^{r,s},\nabla(\lambda))$. As we saw, $\Hom_G(V^{r,s},V^{r',s'})$ has a basis indexed by $((r',s'),(r,s))$-diagrams. Let $D$ be the set of $((r',s'),(r,s))$-diagrams that have no horizontal edges in the top row and whose vertical edges do not cross, and let $(p_d)_{d\in D}$ be the corresponding family of basis elements in $\Hom_G(V^{r,s},V^{r',s'})$. Let $(u_i)_{\in I}$ be a basis of $\Hom_G(V^{r',s'},\nabla(\lambda))$. We have $\Hom_G(V^{r',s'},\nabla(\lambda))\cong S(\lambda^1)\boxtimes S(\lambda^2)$ by Lemma~\ref{lem.diagram}(iii) (with $(r,s)=(r',s')$), $|D|=\binom{r}{t}\binom{s}{t}t!$ and $p_d\ot u_i$ is mapped to $u_i\circ p_d$. So it suffices to show that the elements $u_i\circ p_d$, $d\in D$, $i\in I$, are linearly independent. So assume $\sum_{i,d}a_{id}\,u_i\circ p_d=0$ for certain $a_{id}\in k$. Consider the following diagram $d_0\in D$:
$$
\xy
(-7,2.5)*{d_0=};
(0,0)="a1"*{\bullet};
(6,0)="a2"*{\cdots};
(12,0)="a3"*{\bullet};
(18,0)="a4"*{\bullet};
(24,0)="a5"*{\cdots};
(30,0)="a6"*{\bullet};
(36,0)="a7"*{\bullet};
(42,0)="a8"*{\cdots};
(48,0)="a9"*{\bullet};
(54,0)="a10"*{\bullet};
(60,0)="a11"*{\cdots};
(66,0)="a12"*{\bullet};
(0,5)="b1"*{\bullet};
(6,5)="b2"*{\cdots};
(12,5)="b3"*{\bullet};
(36,5)="b7"*{\bullet};
(42,5)="b8"*{\cdots};
(48,5)="b9"*{\bullet};
{(33,-6);(33,8)**@{--}};
{"a1";"b1"**@{-}};
{"a3";"b3"**@{-}};
{"a7";"b7"**@{-}};
{"a9";"b9"**@{-}};
"a4";"a10"**\crv{(36,-5)};
"a6";"a12"**\crv{(48,-5)};
(6,0)*=<40pt,28pt>{
}*\frm{_\}};%
(6,-7.5)*{\text{\small $r'$ vertices}};
"a5"*=<40pt,28pt>{%
}*\frm{_\}};%
(24,-7.5)*{\text{\small $t$ vertices}};
"a8"*=<40pt,28pt>{%
}*\frm{_\}};%
(42,-7.5)*{\text{\small $s'$ vertices}};
"a11"*=<40pt,28pt>{%
}*\frm{_\}};%
(60,-7.5)*{\text{\small $t$ vertices}}
\endxy 
\quad.
$$
Put $$v_0=e_1\ot\cdots\ot e_{r'}\ot e_{r'+1}\ot\cdots\ot e_r\ot
e^*_{n-s'+1}\ot\cdots\ot e^*_n\ot e^*_{r'+1}\ot\cdots\ot e^*_r$$
and $v_1=e_1\ot\cdots\ot e_{r'}\ot e^*_{n-s'+1}\ot\cdots\ot e^*_n$.
Then we have for $d\in D$ that $p_d(v_0)=v_1$ if $d=d_0$ and $0$ otherwise. It follows that $\sum_ia_{id_0}u_i(v_1)=0$. By Lemma~\ref{lem.diagram}(iii) evaluation at $v_1$ is injective on $\Hom_G(V^{r',s'},\nabla(\lambda))$, so $a_{id_0}=0$ for all $i\in I$. Since we can construct a similar vector for any other $d\in D$ it follows that $a_{id}=0$ for all $i\in I$ and $d\in D$.
\hfil\break
(ii).\ The class of $S(n,r')\ot S(n,s')$-modules $M$ for which this homomorphism is an isomorphism, is closed under taking direct summands and direct sums. The same holds for the class of $S(n,r')\ot S(n,s')$-modules $M$ for which this homomorphism is surjective. By \cite[Lem.~3.4(i)]{D1} every injective $S(n,r')\ot S(n,s')$-module is a direct sum of direct summands of some $S^{\lambda^1}V\boxtimes S^{\lambda^2}V$, $\lambda^1$ and $\lambda^2$ partitions of $r'$ and $s'$ respectively. Furthermore, $\End_{G\times G}(V^{\ot r'}\boxtimes V^{\ot s'})\cong k\Sym_{r',s'}$. So it suffices now to show that the homomorphism is surjective if $M=S^{\lambda^1}V\boxtimes S^{\lambda^2}V$, $\lambda^1,\lambda^2$ as above.

Put $\mc H=\Hom_G(V^{r,s},V^{r',s'})$ and let $f^{(2)}=\Hom_{G\times G}(V^{\ot r'}\boxtimes V^{\ot s'},-)$ be the Schur functor.
Let $0\to M\to N\to P\to 0$ be a short exact sequence of $S(n,r')\ot S(n,s')$-modules with a good filtration. Then we have the following diagram
$$
\hspace{-.2cm}
\xymatrix @R=30pt @C=15pt @M=6pt{
{\mc H}\ot_{k\Sym_{r',s'}}f^{(2)}(M)\ar[r]\ar[d]&
{\mc H}\ot_{k\Sym_{r',s'}}f^{(2)}(N)\ar[r]\ar[d]&
{\mc H}\ot_{k\Sym_{r',s'}}f^{(2)}(P)\ar[r]\ar[d]&0\\
f_{rat}(M)\ar[r]&
f_{rat}(N)\ar[r]&
f_{rat}(P)\ar[r]&0
}
$$
with rows exact, because $f^{(2)}$ is exact and $f_{rat}$ is exact on modules with a good filtration. Here we have used that a $G\times G$-module with a good $G\times G$-filtration, also has a good $G$-filtration; see \cite[II.4.21]{Jan}. We deduce that if the homomorphism in (ii) is surjective for $N$, then it is surjective for $P$. Since the kernel of the canonical epimorphism $V^{\ot r'}\boxtimes V^{\ot s'}\to S^{\lambda^1}V\boxtimes S^{\lambda^2}V$ has a good $G$-filtration by \cite[2.1.14]{D2} applied to $G\times G$, and \cite[2.1.15(ii)(b)]{D2}, we are done.\\
\end{proof}

In the theorem below $f^{(2)}$ is the Schur functor from ${\rm mod}(S(n,r')\ot S(n,s'))$ to ${\rm mod}(k\Sym_{r',s'})$. 
Note that (ii) says that, under the stated condition, the homomorphism in Lemma~\ref{lem.surjective}(ii) is an isomorphism.
\begin{thm}\label{thm.rat_schur} 
Recall that $n\ge r+s$. The following holds.
\begin{enumerate}[{\rm(i)}]
\item For $\lambda=[\lambda^1,\lambda^2]\in\Lambda_{r,s}$ we have
\begin{align*}
f_{rat}(\nabla(\lambda))&\cong {\mc S}(\lambda^1,\lambda^2),\\
f_{rat}(S^{\lambda^1}V\ot S^{\lambda^2}V^*)&\cong {\mc M}(\lambda^1,\lambda^2)\text{, and}\\
f_{rat}(\bigwedge{\Nts}^{\lambda^1}V\ot\bigwedge{\Nts}^{\lambda^2} V^*)&\cong\widetilde{\mc M}(\lambda^1,\lambda^2)\text{\quad if ${\rm char}\,k=0$ or $>\max(|\lambda^1|,|\lambda^2|)$.}
\end{align*}
\item Let $M$ be an $S(n,r')\ot S(n,s')$-module. If $M$ is a direct sum of direct summands of $V^{\ot r'}\boxtimes V^{\ot s'}$ or if $M$ is injective, then
$$f_{rat}(M)\cong {\rm Ind}^{B_{r,s}}_{k\Sym_{r',s'}}f^{(2)}(M).$$ 
\end{enumerate}
\end{thm}

\begin{proof}
Whenever $\lambda=[\lambda^1,\lambda^2]\in\Lambda_{r,s}$ we assume $|\lambda^1|=r'$ and $|\lambda^2|=s'$. If we give $\Hom_G(V^{r',s'},\nabla(\lambda))$ the $k\Sym_{r',s'}$-module structure coming from the action of $\Sym_{r',s'}$ on $V^{r',s'}$ by place permutations, then the isomorphisms in \eqref{eq.rat_schur_iso} are $\Sym_{r',s'}$-equivariant. Now Lemma~\ref{lem.surjective}(i) and the isomorphism \eqref{eq.I_tiso} give us an epimorphism
$I_t\ot_{k\Sym_{r',s'}}(S(\lambda^1)\boxtimes S(\lambda^2))\to f_{rat}(\nabla(\lambda))$. The image of a nonzero homomorphism from $V^{r,s}$ to $\nabla(\lambda)$ must contain $L(\lambda)$ and therefore have $\lambda$ as a weight. The image of a homomorphism in $\varphi(I_{t,1})$ does not have $\lambda$ as a weight, since $\varphi(I_{t,1})$ has a basis of homomorphisms whose image lies is a submodule of $V^{r',s'}$ which is isomorphic to $V^{r'-1,s'-1}$. So, by \eqref{eq.Z_tiso} and the definition of $S(\lambda^1,\lambda^2)$, we obtain an epimorphism $S(\lambda^1,\lambda^2)\to f_{rat}(\nabla(\lambda))$. By Lemma~\ref{lem.dimension} this must be an isomorphism.

Let $M$ be an $S(n,r')\ot S(n,s')$-module. Lemma~\ref{lem.surjective}(ii) and the isomorphism $\varphi$ give us a homomorphism
$${\rm Ind}^{B_{r,s}}_{k\Sym_{r',s'}}f^{(2)}(M)\to f_{rat}(M)\eqno(*)$$
which is an isomorphism if $M$ is a direct sum of direct summands of \hbox{$V^{\ot r'}\boxtimes V^{\ot s'}$} and surjective for $M$ injective.
Note that $S^{\lambda^1}V\boxtimes S^{\lambda^2}V=S^{\lambda^1}V\ot S^{\lambda^2}V^*$ as $G$-modules and similar for exterior powers.
So we obtain an epimorphism $\mc M(\lambda^1,\lambda^2)\to f_{rat}(S^{\lambda^1}V\ot S^{\lambda^2}V^*)$ and a homomorphism $\widetilde{\mc M}(\lambda^1,\lambda^2)\to$\break
$f_{rat}(\bigwedge{\Nts}^{\lambda^1}V\ot\bigwedge{\Nts}^{\lambda^2} V^*)$, since $f(S^\mu V)=M(\mu)$ and $f(\bigwedge{\Nts}^\mu V)=k_{\rm sg}\ot M(\mu)$ by \cite[Lemma~3.5]{D1}. If ${\rm char}\,k=0$ or $>\max(r',s')$, then $S(n,r')\ot S(n,s')$ is semisimple, so every $S(n,r')\ot S(n,s')$-module is a direct sum of direct summands of $V^{\ot r'}\boxtimes V^{\ot s'}$ and (*) is an isomorphism for every $S(n,r')\ot S(n,s')$-module $M$. In particular, we have the third isomorphism in (i).

It remains to show that the epimorphism ${\mc M}(\lambda^1,\lambda^2)\to f_{rat}(S^{\lambda^1}V\ot S^{\lambda^2}V^*)$ is an isomorphism. Since (*) is an isomorphism if ${\rm char}\,k=0$ it suffices to show that the dimensions of $f_{rat}(S^{\lambda^1}V\ot S^{\lambda^2}V^*)$ and $\mc M(\lambda^1,\lambda^2)$ are independent of the characteristic. The dimension of $f_{rat}(S^{\lambda^1}V\ot S^{\lambda^2}V^*)$ is independent of the characteristic, since, by \cite[Prop.~A.2.2(ii)]{D2}, it only depends on the formal characters of the $G$-modules $V^{r,s}$ and $S^{\lambda^1}V\ot S^{\lambda^2}V^*$ (and these are independent of the characteristic). That $\mc M(\lambda^1,\lambda^2)$ has dimension independent of the characteristic follows from Proposition~\ref{prop.filtration}, the fact that $M(\lambda^1)\boxtimes M(\lambda^2)$ is self dual and the following well-known fact.

{\it Let $H$ be a finite group, let $N$ be a permutation module for $H$ over $k$ with $H$-stable basis $S$. Then the dimension of $N^H$ is equal to the number of $H$-orbits in $S$.}

We have now proved the second isomorphism in (i) and we have also proved (ii), since every injective $S(n,r')\ot S(n,s')$-module is a direct sum of direct summands of some $S^{\lambda^1}V\boxtimes S^{\lambda^2}V$, $\lambda^1,\lambda^2$ partitions of $r'$ resp. $s'$.
\end{proof}

For $\lambda\in\Lambda_{r,s}$ with $\lambda^1,\lambda^2$ $p$-regular and $\lambda^1,\lambda^2\ne\emptyset$ in case $r=s\ge1$ and $\delta=0$, we denote the projective cover of the irreducible $B_{r,s}$-module $\mc D(\lambda^1,\lambda^2)$ by $\mc P(\lambda^1,\lambda^2)$.

\begin{prop}\label{prop.rat_mult} 
Let $\lambda=[\lambda^1,\lambda^2],\mu=[\mu^1,\mu^2]\in\Lambda_{r,s}$. Then $T(\lambda)$ is a direct summand of the $G$-module $V^{r,s}$ if and only if $\lambda^1,\lambda^2$ are $p$-regular and $\lambda^1,\lambda^2\ne\emptyset$ in case $r=s\ge1$ and $\delta=0$. Now assume that $\lambda$ satisfies these conditions. Then
\begin{enumerate}[{\rm(i)}]
\item $f_{rat}(T(\lambda))={\mc P}(\lambda^1,\lambda^2)$.
\item The multiplicity of $T(\lambda)$ in $V^{r,s}$ is $\dim{\mc D}(\lambda^1,\lambda^2)$.
\item The decomposition number $[{\mc S}(\mu^1,\mu^2):{\mc D}(\lambda^1,\lambda^2)]$ equals the $\Delta$-filtration multiplicity $(T(\lambda):\Delta(\mu))$ and the $\nabla$-filtration multiplicity $(T(\lambda):\nabla(\mu))$.
\end{enumerate}
\end{prop}

\begin{proof}
Let $\Omega$ be the set of all partitions satisfying the stated conditions. The rational Schur functor $f_{rat}$ induces a category equivalence between the direct sums of direct summands of the $G$-module $V^{r,s}$ and the projective $B_{r,s}$-modules; see e.g. \cite[Prop~2.1(c)]{ARS}. Clearly, the number of isomorphism classes of indecomposable $B_{r,s}$-projectives is equal to $|\Omega|$. So, to prove the first assertion, it suffices to show that for each $\lambda\in\Omega$, $T(\lambda)$ is a direct summand of $V^{r,s}$. By Lemma~\ref{lem.directsummand} we may assume that $|\lambda^1|=r$ and $|\lambda^2|=s$. The indecomposable tilting module $T(\lambda^1)$ is a direct summand of $V^{\ot r}$, for example by \cite[Sect.~4.3, (1) and (4)]{D2}. Twisting with the inverse transpose, the same argument gives that $T(-\breve{\,\lambda^2)}$ is a direct summand of $(V^*)^{\ot s}$. So the tilting module $T(\lambda^1)\ot T(-\breve{\,\lambda^2)}$ is a direct summand of $V^{r,s}$. Since $T(\lambda^1)\ot T(-\breve{\,\lambda^2)}$ has highest weight $\lambda$, it has $T(\lambda)$ as a direct summand. It follows that $T(\lambda)$ occurs as a component of $V^{r,s}$.

Let $\lambda\in\Omega$. By Theorem~\ref{thm.rat_schur}(i) $f_{rat}(T(\lambda))$ surjects onto $f_{rat}(\nabla(\lambda))=\mc S(\lambda^1,\lambda^2)$. But $\mc S(\lambda^1,\lambda^2)$ surjects onto $\mc D(\lambda^1,\lambda^2)$. This proves (i), and (ii) is now also clear, since this multiplicity (as an indecomposable direct summand) is equal to the multiplicity of $\mc P(\lambda^1,\lambda^2)$ in $B_{r,s}$. We have $g_{rat}(f_{rat}(M))\cong M$ canonically for $M=V^{r,s}$ and therefore also for $M=T(\lambda)$. By \eqref{eq.adjointiso} we have $\Hom_G(T(\lambda),M)\cong\Hom_{B_{r,s}}(\mc P(\lambda^1,\lambda^2),f_{rat}(M))$ for every $S(n;r,s)$-module $M$. It follows that $[\mc S(\mu^1,\mu^2):\mc D(\lambda^1,\lambda^2)]=\dim\Hom_{B_{r,s}}(\mc P(\lambda^1,\lambda^2),\mc S(\mu^1,\mu^2))=\dim\Hom_G(T(\lambda),\nabla(\mu))=(T(\lambda):\Delta(\mu))=(T(\lambda):\nabla(\mu))$.
\end{proof}

From Theorem~\ref{thm.filt_mult} and Proposition~\ref{prop.rat_mult} we now obtain the following corollary.
\begin{cornn}
Let $\lambda^1,\lambda^2,\mu^1,\mu^2$ be partitions with $r-|\lambda^1|=s-|\lambda^2|\ge0$ and $r-|\mu^1|=s-|\mu^2|\ge0$ and assume that $\lambda^1,\lambda^2\ne\emptyset$ if $r=s\ge1$ and $\delta=0$. Assume also that $\lambda^h_1+l(\lambda^h)\le p$ for all $h\in\{1,2\}$. Choose $n\ge r+s$ such that $n=\delta$ mod $p$. Put $\lambda=[\lambda^1,\lambda^2]$, $\mu=[\mu^1,\mu^2]$ and choose $s_1,s_2\in\{1,\ldots,\min(n,p)\}$ with $s_1+s_2\le n$ and $\lambda\in\Lambda(s_1,s_2)$. Then
$$[\mc S(\mu^1,\mu^2):\mc D(\lambda^1,\lambda^2)]=
\begin{cases}
1\text{\ if $\mu\preceq\lambda$ and $c_{\lambda\mu}$ is oriented,}\\
0\text{\ otherwise.}
\end{cases}
$$
\end{cornn}

\begin{rems}\label{rems.rat_schur}
1.\ From Proposition~\ref{prop.rat_mult} it is clear that when $p>\max(r,s)$ and, in case $r=s$, $n\ne0$ in $k$, then $V^{r,s}$ is a full tilting module for $S(n;r,s)$ and the walled Brauer algebra $B_{r,s}(n)$ is the Ringel dual, see e.g. \cite[Appendix A4]{D2}, of the rational Schur algebra $S(n;r,s)$.\\
2.\ Let $f_{rat}^{r',s'}$ be the rational Schur functor from ${\rm mod}(S(n;r',s'))$ to ${\rm mod}(B_{r',s'})$ and let $M$ be a $G$-module which has a filtration with sections isomorphic to some $\nabla(\lambda)$ with $|\lambda^1|=r'$ and $|\lambda^2|=s'$. Then $$f_{rat}(M)\cong Z_t\ot_{k\Sym_{r',s'}}f_{rat}^{r',s'}(M).$$
This is shown as in the case of the symplectic group, see \cite[Rem~2.1.1]{DT}.\\
3.\ Put $\pi_{r,s}=\{\lambda=[\lambda^1,\lambda^2]\in\Lambda_{r,s}\,|\,|\lambda^1|<r, |\lambda^2|<s\}$. Let $M$ be an $S(n,r)\ot S(n,s)$-module and put $N=O_{\pi_{r,s}}(M)$. By \cite[Prop.~A2.2(v), Lem.~A3.1]{D2} $N$ has a filtration with sections $\nabla(\lambda)$, $\lambda\in\pi_{r,s}$, and $M/N$ has a filtration with sections $\nabla(\lambda)$, $\lambda\in\Lambda_{r,s}$ with $|\lambda^1|=r$ (and $|\lambda^2|=s$). Note that $M/N=\nabla(\lambda)$ if $M=\nabla(\lambda^1)\boxtimes\nabla(\lambda^2)$. Now we can form the diagram
\begin{equation*}
\xymatrix @R=30pt @C=15pt @M=6pt{
f^{(2)}(M)\ar[r]\ar[d]& M_{(\varpi_r,-\breve{\varpi}_s)}\ar[d]\\
f_{rat}(M/N)\ar[r]& (M/N)_{[\varpi_r,\varpi_s]}
}
\end{equation*}
in the same way as \eqref{eq.rat_schur_iso} and by a proof very similar to that of Lemma~\ref{lem.diagram}(iii) we show that all maps are isomorphisms.
\end{rems}

\subsection{Further results on the rational Schur functor}\label{ss.further_results}
\begin{lem}\label{lem.tensor}
Let $M$ be an $G$-module. Then the canonical homomorphism
$$V^{r,s}\ot_{B_{r,s}}\Hom_G(V^{r,s},M)\to M$$ given by function application is an isomorphism if $M$ is a direct sum of direct summands of $V^{r,s}$ or if $r=s\ge2$ and $M=k$.
\end{lem}
\begin{proof}
That the canonical homomorphism is an isomorphism under the first condition is obvious, since $\End_G(V^{r,s})\cong B_{r,s}$. So assume $r=s\ge2$ and $M=k$. Since the homomorphism is always surjective and, as vector spaces,
$$V^{r,r}\ot_{B_{r,r}}\Hom_G(V^{r,r},k)\cong\Hom_{B_{r,r}}(\Hom_G(V^{r,r},k),V^{r,r})^*$$
by \cite[Lemma~3.60]{Rot} and the self-duality of $V^{r,r}$ as $B_{r,r}$-module, it suffices to show that $\Hom_{B_{r,r}}(\Hom_G(V^{r,r},k),V^{r,r})$ is one-dimensional. 
Recall that $\Hom_G(V^{r,r},k)$ is a left $B_{r,r}$-module by means of the standard anti-automorphism $\iota$ of $B_{r,r}$. It has a basis indexed by $((0,0),(r,r))$-diagrams and it is generated as a $k\Sym_{r,r}$-module by the homomorphism $P$ corresponding to the $((0,0),(r,r))$-diagram
$$
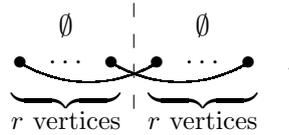
\xy
(0,0)="a1"*{\bullet};
(6,0)="a2"*{\cdots};
(12,0)="a3"*{\bullet};
(18,0)="a4"*{\bullet};
(24,0)="a5"*{\cdots};
(30,0)="a6"*{\bullet};
(6,5)="b2"*{\emptyset};
(24,5)="b8"*{\emptyset};
{(15,-6);(15,8)**@{--}};
"a1";"a4"**\crv{(9,-5)};
"a3";"a6"**\crv{(21,-5)};
"a2"*=<40pt,28pt>{
}*\frm{_\}};%
(6,-7.5)*{\text{\small $r$ vertices}};
"a5"*=<40pt,28pt>{%
}*\frm{_\}};%
(24,-7.5)*{\text{\small $r$ vertices}};
\endxy
\quad.
$$

It follows that any $B_{r,r}$-homomorphism from $\Hom_G(V^{r,r},k)$ to $V^{r,r}$ is determined by its image of $P$. One easily checks that $P\circ\iota(d)=P$, where $d\in B_{r,r}$ is given by
$$
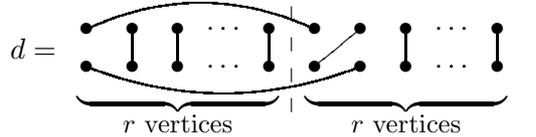
\xy
(-7,2.5)*{d=};
(0,0)="a1"*{\bullet};
(6,0)="a2"*{\bullet};
(12,0)="a3"*{\bullet};
(18,0)="a4"*{\cdots};
(24,0)="a5"*{\bullet};
(30,0)="a6"*{\bullet};
(36,0)="a7"*{\bullet};
(42,0)="a8"*{\bullet};
(48,0)="a9"*{\cdots};
(54,0)="a10"*{\bullet};
(0,5)="b1"*{\bullet};
(6,5)="b2"*{\bullet};
(12,5)="b3"*{\bullet};
(18,5)="b4"*{\cdots};
(24,5)="b5"*{\bullet};
(30,5)="b6"*{\bullet};
(36,5)="b7"*{\bullet};
(42,5)="b8"*{\bullet};
(48,5)="b9"*{\cdots};
(54,5)="b10"*{\bullet};
{(27,-6);(27,8)**@{--}};
{"a2";"b2"**@{-}};
{"a3";"b3"**@{-}};
{"a5";"b5"**@{-}};
{"a6";"b7"**@{-}};
{"a8";"b8"**@{-}};
{"a10";"b10"**@{-}};
"a1";"a7"**\crv{(18,-7)};
"b1";"b6"**\crv{(15,12)};
"a3"*=<75pt,28pt>{
}*\frm{_\}};%
(12,-7.5)*{\text{\small $r$ vertices}};
"a8"*=<75pt,28pt>{%
}*\frm{_\}};%
(42,-7.5)*{\text{\small $r$ vertices}};
\endxy
\quad.
$$
Therefore the image of $P$ under such a homomorphism must lie in $d\cdot V^{r,r}=\big\{\sum_{i=1}^ne_i\ot u\ot e^*_i\ot v\,|\,u\in V^{\ot(r-1)}, v\in (V^*)^{\ot(r-1)}\big\}$. But then it must lie in the $\pi$-conjugate of this subspace for any $\pi$ in the diagonal copy of $\Sym_r$ in $\Sym_{r,r}$, since such a $\pi$ fixes $P$. We conclude that the image of $P$ under any $B_{r,r}$-homomorphism from $\Hom_G(V^{r,r},k)$ to $V^{r,r}$ must be a scalar multiple of $\sum_{i_1,\ldots,i_r=1}^ne_{i_1}\ot\cdots\ot e_{i_r}\ot e^*_{i_1}\ot\cdots\ot e^*_{i_r}$.
\end{proof}

In the proposition below $g^{(2)}$ is the inverse Schur functor from ${\rm mod}(k\Sym_{r',s'})$ to ${\rm mod}(S(n,r')\ot S(n,s'))$ given by $g^{(2)}(M)=(V^{\ot r'}\boxtimes V^{\ot s'})\ot_{k\Sym_{r',s'}}M$.
\begin{prop}\label{prop.ratyoung}\
\begin{enumerate}[{\rm(i)}]
\item If $n=0$ in $k$ and $r'=s'=0$, assume $r\ge2$. Then we have
$$g_{rat}({\rm Ind}^{B_{r,s}}_{k\Sym_{r',s'}} N)\cong g^{(2)}(N)$$
as $G$-modules, for every $k\Sym_{r',s'}$-module $N$.
\item Let $\lambda=[\lambda^1,\lambda^2]\in\Lambda_{r,s}$. If $\lambda^1=\lambda^2=\emptyset$ and $n=0$ in $k$, then assume $r=s\ge2$. Then $g_{rat}({\mc M}(\lambda^1,\lambda^2))\cong S^{\lambda^1}V\ot S^{\lambda^2}V^*$ and if ${\rm char}\,k\ne2$, then $g_{rat}(\widetilde{\mc M}(\lambda^1,\lambda^2))\cong \bigwedge{\Nts}^{\lambda^1}V\ot\bigwedge{\Nts}^{\lambda^2}V^*$.
\item Let $\lambda=[\lambda^1,\lambda^2]\in\Lambda_{r,s}$. The $G$-module $S^{\lambda^1}V\ot S^{\lambda^2}V^*$ has a unique indecomposable summand $J(\lambda)$ in which $\nabla(\lambda)$ has filtration multiplicity $>0$ and this multiplicity is equal to $1$. Every summand of ${\mc M}(\lambda^1,\lambda^2)$ has a Specht filtration and $f_{rat}(J(\lambda))\cong{\mc Y}(\lambda^1,\lambda^2)$.
\end{enumerate}
\end{prop}
\begin{proof}
(i).\ Since ${\rm Ind}^{B_{r,s}}_{k\Sym_{r',s'}} N\cong\Hom_{G}(V^{r,s},V^{r',s'})\ot_{k\Sym_{r',s'}}N$, this follows from Lemmas~\ref{lem.directsummand} and \ref{lem.tensor} applied to $V^{r',s'}$.\\
(ii).\ If we take $(r',s')=(|\lambda^1|,|\lambda^2|)$ in (i), then we get $g_{rat}({\mc M}(\lambda^1,\lambda^2))\cong g^{(2)}(M(\lambda^1)\boxtimes M(\lambda^2))$ and $g_{rat}({\mc M}(\lambda^1,\lambda^2))\cong g^{(2)}(k_{\rm sg}\ot M(\lambda^1)\boxtimes M(\lambda^2))$. One easily verifies that $g^{(2)}(M(\lambda^1)\boxtimes M(\lambda^2))\cong S^{\lambda^1}V\boxtimes S^{\lambda^2}V$ and, in case ${\rm char}\,k\ne2$, $g^{(2)}(k_{\rm sg}\ot M(\lambda^1)\boxtimes M(\lambda^2))\cong\bigwedge{\Nts}^{\lambda^1}V\boxtimes\bigwedge{\Nts}^{\lambda^2}V$.\\
(iii).\ Put $(r',s')=(|\lambda^1|,|\lambda^2|)$. The filtration multiplicity of $\nabla(\lambda^1)\boxtimes\nabla(\lambda^2)$ in $S^{\lambda^1}V\ot S^{\lambda^2}V$ is 1 and if, for $\mu=[\mu^1,\mu^2]\in\Lambda_{r,s}$, $\nabla(\nu)$ has filtration multiplicity $>0$ in $\nabla(\mu^1)\boxtimes\nabla(\mu^2)$, then either $\nu=\mu$ and the multiplicity is 1 or $|\nu^h|<|\mu^h|$ for all $h\in\{1,2\}$ as one can easily deduce from Lemma~\ref{lem.diagram}(ii). We conclude that the filtration multiplicity of $\nabla(\lambda)$ in $S^{\lambda^1}V\ot S^{\lambda^2}V^*$ is 1. A direct summand of a module with a good filtration has a good filtration. So, by the Krull-Schmidt theorem, there is a unique indecomposable summand $J(\lambda)$ in which $\nabla(\lambda)$ has filtration multiplicity $>0$. This proves the first assertion. If $\lambda^1=\lambda^2=\emptyset$, then $S^{\lambda^1}V\ot S^{\lambda^2}V^*=k$, $t=r=s$, $Z_r=I_r$. An argument very similar to that of the proof of Lemma~\ref{lem.tensor} shows that $\End_{B_{r,s}}(I_r)=k$, i.e $I_r$ is indecomposable. So ${\mc S}(\lambda^1,\lambda^2)={\mc M}(\lambda^1,\lambda^2)={\mc Y}(\lambda^1,\lambda^2)=I_r$ and the second assertion is obvious.
Now assume $(\lambda^1,\lambda^2)\ne(\emptyset,\emptyset)$. By (ii) and Theorem~\ref{thm.rat_schur}(i) we have $g_{rat}(f_{rat}(M))\cong M$ canonically for every direct summand of $S^{\lambda^1}V\ot S^{\lambda^2}V^*$ and $f_{rat}(g_{rat}(N))\cong N$ canonically for every direct summand $N$ of ${\mc M}(\lambda^1,\lambda^2)$. In particular, every direct summand of ${\mc M}(\lambda^1,\lambda^2)$ has a Specht filtration.

Now let $I(\lambda^h)\subseteq S^{\lambda^h}V$, $h=1,2$, be the $S(n,r')$ resp $S(n,s')$-injective hull of $\nabla(\lambda^h)$, $h=1,2$. Then $I(\lambda^1,\lambda^2)=I(\lambda^1)\ot I(\lambda^2)\subseteq S^{\lambda^1}V\boxtimes S^{\lambda^2}V$ is the $S(n,r')\ot S(n,s')$-injective hull of $\nabla(\lambda^1)\boxtimes\nabla(\lambda^2)$. By \cite[3.6]{D1} we have $f(I(\nu))=Y(\nu)$. Put $\pi=\pi_{r',s'}=\{\mu=[\mu^1,\mu^2]\in\Lambda_{r,s}\,|\,|\mu^1|<r', |\mu^2|<s'\}$. By Remarks~\ref{rems.rat_schur}, 2 and 3 we have $f_{rat}(I(\lambda^1,\lambda^2)/O_\pi(I(\lambda^1,\lambda^2)))\cong Z_t\ot_{k\Sym_{r',s'}}(Y(\lambda^1)\boxtimes Y(\lambda^2))$. As in \cite[Prop.~3]{HarPag} $Z_t\ot_{k\Sym_{r',s'}}(Y(\lambda^1)\boxtimes Y(\lambda^2))$ is indecomposable. Since $I(\lambda^1,\lambda^2)/O_\pi(I(\lambda^1,\lambda^2))$ has a good filtration, it must also be indecomposable. Write $I(\lambda^1,\lambda^2)=\bigoplus_{i=1}^l J_i$ with each $J_i$ an indecomposable $G$-module. Then $I(\lambda^1,\lambda^2)/O_\pi(I(\lambda^1,\lambda^2))\cong\bigoplus_{i=1}^l J_i/O_\pi(J_i)$. So there is a unique $j\in\{1,\cdots,l\}$ such that $J_j/O_\pi(J_j)\cong I(\lambda^1,\lambda^2)/O_\pi(I(\lambda^1,\lambda^2))$ and $J_i\subseteq O_\pi(I(\lambda^1,\lambda^2))$ for all $i\ne j$. Clearly we must have $J_j\cong J(\lambda)$. Furthermore, since the kernel of $J(\lambda)\to J(\lambda)/O_\pi(J(\lambda))$ has a good filtration, we have that $f_{rat}(J(\lambda))$ surjects onto $Z_t\ot_{k\Sym_{r',s'}}Y(\lambda^1,\lambda^2)$. So $f_{rat}(J(\lambda))\cong{\mc Y}(\lambda^1,\lambda^2)$.
\end{proof}

\begin{rems}
1.\ The rational Schur coalgebra is $A(n;r,s)=O_{\Lambda_{r,s}}(k[G])$, where the action of $G$ on $k[G]$ comes from right multiplication in $G$; see \cite[A.14]{Jan} for the generalities.
We can also let $G$ act on $k[G]$ using left multiplication and the transpose map. For this action we also have $A(n;r,s)=O_{\Lambda_{r,s}}(k[G])$. Now the two actions on $k[G]$ are isomorphic via the comorphism of the transpose map. This isomorphism restricts to an isomorphism of the two actions on $A(n;r,s)$. With the left multiplication action (via the transpose map), $A(n;r,s)$ is $S(n;r,s)^\circ$ where $S(n;r,s)$ has the left multiplication action of $G$ which corresponds to the left regular action of $S(n;r,s)$. Now give $A(n;r,s)$ the right multiplication action and $S(n;r,s)$ the left multiplication action. 
Recall also that $V^{r,s}$ is self-dual as a $B_{r,s}$-module. Then
\begin{align*}f_{rat}(A(n;r,s))&=\Hom_{G}(V^{r,s},S(n;r,s)^\circ)\cong\Hom_{G}(S(n;r,s),V^{r,s})\cong V^{r,s}\text{\ and}\\
g_{rat}(V^{r,s})&=\tilde g_{rat}(V^{r,s})^\circ=\End_{B_{r,s}}(V^{r,s})^\circ=S(n;r,s)^\circ\cong A(n;r,s). 
\end{align*}
The class of $S(n;r,s)$-modules $M$ for which $g_{rat}(f_{rat}(M))\cong M$ canonically, is closed under taking direct summands and direct sums. In particular it contains the injective $S(n;r,s)$-modules, since, by the above, it contains $A(n;r,s)$. For the same reason the class of $B_{r,s}$-modules $N$ for which $f_{rat}(g_{rat}(N))\cong N$ canonically, contains the projective $B_{r,s}$-modules.\\
2.\ The results for the rational Schur functor look more like the results in \cite{DT} for the orthogonal Schur functor (Sect~2) than like those for the symplectic Schur functor (Sect~4).
This is because we work with $B_{r,s}(\delta)$ as a subalgebra of $B_{r+s}(\delta)$. There is also a ``symplectic Brauer algebra" $\widetilde {B}_{r+s}(\delta)$, see e.g. \cite[p 871]{Br}, \cite{Wen} or \cite[Sect.~3]{T0}.
Furthermore, there is an isomorphism $B_{r+s}(-\delta)\stackrel{\sim}{\to}\widetilde{B}_{r+s}(\delta)$ (*) which sends each of the $r+s$ standard generators of $B_{r+s}(-\delta)$ to the negative of the corresponding standard generator of $\widetilde{B}_{r+s}(\delta)$, see the proof of \cite[Cor~3.5]{Wen}.
One can define a walled subalgebra $\widetilde {B}_{r,s}(\delta)$ of $\widetilde {B}_{r+s}(\delta)$ in precisely the same way as $B_{r,s}(\delta)$ was defined as a subalgebra of $B_{r+s}(\delta)$. Now one can check that two ``walled diagrams" in $\widetilde {B}_{r,s}(\delta)$ multiply precisely as in $B_{r,s}(\delta)$, i.e. their ``symplectic sign" equals $1$.
It is enough to check this on generators and with $\delta$ specialised to $2m$, and one can also easily deduce it from the description of the sign in \cite[Sect.~3]{T0}. 
So we have $\widetilde {B}_{r,s}(\delta)=B_{r,s}(\delta)$ and the isomorphism (*) restricts to an isomorphism $\theta:B_{r,s}(-\delta)\stackrel{\sim}{\to}B_{r,s}(\delta)$.
Now we could let $B_{r,s}(-n)$ act on $V^{r,s}$ via this isomorphism and then we could define another version of the rational Schur functor ${\rm mod}(S(n;r,s))\to{\rm mod}(B_{r,s}(-n))$ for which the results would look like those for the symplectic Schur functor. However, these results can also be obtained from the present results by applying the equivalence of categories ${\rm mod}(B_{r,s}(n))\stackrel{\sim}{\to}{\rm mod}(B_{r,s}(-n))$ given by $\theta$.
For example, when, for $M$ an $S(n,r')\ot S(n,s')$-module, we turn ${\rm Ind}_{k\Sym_{r',s'}}^{B_{r,s}(n)}M$ into a $B_{r,s}(-n)$-module via $\theta$, then we obtain ${\rm Ind}_{k\Sym_{r',s'}}^{B_{r,s}(-n)}k_{\rm sg}\ot M$.
\end{rems}


\end{document}